\documentclass[psamsfonts, reqno]{amsart}

\usepackage{amssymb,amsfonts}
\usepackage[all,arc]{xy}
\usepackage{enumerate}
\usepackage{mathrsfs}
\usepackage{slashbox}
\usepackage{geometry}
\usepackage{hyperref}
\hypersetup{
    colorlinks=true,
    linkcolor=blue,
    filecolor=blue,
    urlcolor=blue,
    citecolor=red,
}

\newtheorem{theorem}{Theorem}[section]

\newtheorem{lemma}[theorem]{Lemma}

\theoremstyle{definition}
\newtheorem{definition}[theorem]{Definition}
\newtheorem{remark}{Remark}[section]

\numberwithin{equation}{section}

\title[HEAT CONDUCTING MHD EQUATIONS]{Strong solutions for three-dimensional nonhomogeneous incompressible Heat Conducting Magnetohydrodynamic Equations with vacuum}

\date{}
\email{\color{blue} lhymaths@zzu.edu.cn}
\begin{document}

\maketitle

\centerline{\scshape Huanyuan Li \footnote{This work was supported by National Natural Science Foundation of China (Grant No. 12001495).}}
\medskip
{\footnotesize
 \centerline{School of Mathematics and Statistics, Zhengzhou University}

} 

\begin{abstract}
This paper is concerned with a Cauchy problem for the three-dimensional (3D) nonhomogeneous incompressible heat conducting magnetohydrodynamic (MHD) equations  in the whole space. First of all, we establish a weak Serrin-type blowup criterion for the strong solutions. It is shown that for the Cauchy problem of the 3D nonhomogeneous heat conducting MHD equations, the strong solution exists globally if the velocity satisfies the weak Serrin's condition. In particular, this criterion is independent of the absolute temperature and magnetic field. Then as an immediate application, we prove the global existence and uniqueness of strong solution to the 3D nonhomogeneous heat conducting MHD equations under some smallness condition on the initial data. In addition, the initial vacuum is allowed.

\vspace{2mm}

\end{abstract}

\section{Introduction and main results}

The time evolution of a three-dimensional nonhomogeneous incompressible and heat conducting magnetohydrodynamic (MHD for short) fluid is governed by the following nonhomogeneous heat conducting MHD system:
\begin{equation} \label{MHD}
\left\{
\begin{aligned}
&\partial_t\rho + \mathrm{div}(\rho \mathbf{u}) = 0, \\
&\partial_t(\rho \mathbf{u}) + \mathrm{div} (\rho \mathbf{u}\otimes \mathbf{u}) - \mathrm{div}(2\mu\mathfrak{D}(\mathbf{u})) +\nabla P = (\mathbf{H}\cdot \nabla) \mathbf{H}, \\
& c_v [\partial_t (\rho \theta) + \mathrm{div} (\rho \mathbf{u} \theta)]-\kappa \Delta \theta = 2\mu |\mathfrak{D} (\mathbf{u})|^2 + \nu |\nabla \times \mathbf{H}|^2, \\
& \partial_t \mathbf{H} + (\mathbf{u} \cdot \nabla) \mathbf{H} - (\mathbf{H} \cdot \nabla) \mathbf{u} = \nu \Delta \mathbf{H}, \\
&\mathrm{div} \mathbf{u}= \mathrm{div} \mathbf{H} = 0,
\end{aligned}
\right.
\end{equation}
where $t \ge 0$ stands for the time and $x \in \mathbb{R}^3$ for the spatial coordinate. $\rho,\mathbf{u}=(u^1, u^2, u^3), P,$ $\theta$ and $\mathbf{H}=(H^1, H^2, H^3)$ denote the fluid density, velocity, pressure, absolute temperature and magnetic field, respectively. The positive constants $\mu, c_v, \kappa$ and $\nu$ are the viscosity coefficient, heat capacity, heat conductivity coefficient and magnetic diffusive coefficient, respectively.
$$ \mathfrak{D}(\mathbf{u}) = \frac{1}{2} \Big[ \nabla \mathbf{u} + (\nabla \mathbf{u})^T \Big] $$
is the deformation tensor, where $\nabla \mathbf{u}$ is the gradient matrix $\left(\partial u^i / \partial x_j \right)_{ij}$ and $(\nabla \mathbf{u})^T$ is its transpose, and $\nabla \times \mathbf{H}$ is the curl of the magnetic field $\mathbf{H}$.

In this paper, we will study the Cauchy problem of the Eqs. \eqref{MHD} with the initial conditions
\begin{equation} \label{initial}
(\rho,\mathbf{u}, \mathbf{H}, \theta)(x, 0) = (\rho_0, \mathbf{u}_0, \mathbf{H}_0, \theta_0)(x), \quad \quad x \in \mathbb{R}^3,
\end{equation}
and the far field behavior conditions (in some weak sense)
\begin{equation} \label{boundary}
(\rho, \mathbf{u}, \mathbf{H}, \theta) (x, t) \rightarrow (0, \mathbf{0}, \mathbf{0},0), \quad \quad \text{as} \quad |x| \rightarrow + \infty.
\end{equation}

Magnetohydrodynamics is the study of the interaction between magnetic field and moving conducting fluids. It is one of the important macroscopic fluid models, usually arising in science and engineering with a wide range of applications. Examples of such a magneto-fluids include hot ionised gases (plasmas), liquid metals or strong electrolytes. Because of the physical description of magneto-fluids dynamics, the nonhomogeneous incompressible MHD system \eqref{MHD} is a combination of the nonhomogeneous Navier-Stokes equations of fluid mechanics and the Maxwell equations of electromagnetism. The concept behind MHD is that the magnetic field can induce currents in a moving conducting fluid, which in turn polarizes the fluid and changes the magnetic field itself. One of the important issues is to understand the nature of this coupling between fluids and magnetic fields. We refer to \cite{David} for more background and applications of MHD.

The mathematical studies on the nonhomogeneous incompressible fluids attract a lot of attentions due to their physical importance, mathematical challenge and widespread applications. Let us briefly give a short survey on the works of nonhomogeneous fluids which are related to our results in this paper.

When we don't take account of the equation $\eqref{MHD}_3$ for temperature, $\eqref{MHD}$ reduces to the nonhomogeneous incompressible MHD equations. For this system, when the initial density has a positive lower bound, Gerbeau, Le Bris \cite{GLebris} and Desjardins, Le Bris \cite{DLebris} studied the global existence of weak solutions of finite energy in the whole space or in the torus respectively. Chen et. al. \cite{ChenL} proved a global solution for the initial data belonging to critical Besov spaces. See also \cite{Bie} for related improvement. Besides, Chen et. al. \cite{ChenG} showed global well-posedness to the 3D Cauchy problem for discontinuous initial density. On the other hand, in the presence of vacuum, Chen et. al. \cite{ChenT} obtained the local existence of strong solutions to the 3D Cauchy problem under some compatibility condition on the initial data. With the help of a Sobolev inequality of logarithmic type, Huang and Wang \cite{HuangW} showed the global existence of strong solution for general initial data. 

When we study the motion in the absence of magnetic field, namely, $\mathbf{H} \equiv 0$, $\eqref{MHD}$ reduces to the nonhomogeneous heat conducting Navier-Stokes equations. Under compatibility conditions for the initial data, Zhong \cite{Zhong2} showed a Serrin-type blow-up criterion and proved global strong solutions with vacuum for small initial data, which extended the local result obtained by Cho and Kim \cite{ChoK} to a global one. By employing certain time-weighted a priori estimates, they showed that strong solutions exist globally provided that some smallness condition holds true. Meanwhile, Wang et al. \cite{WangYu} studied a three-dimensional initial boundary value problem with the general external force and obtained global existence of strong solutions under the assumption that the initial density is suitably small. Very recently, combining delicate energy estimates and a logarithmic interpolation inequality, the author established the global existence and uniqueness of strong solutions to the 2D Cauchy problem with large initial data and non-vacuum density at infinity.

Let us go back to the heat conducting MHD system \eqref{MHD}. The local existence of a unique strong solution to the system \eqref{MHD} with vacuum under some compatibility conditions was proved by Wu \cite{Wu}. For 3D initial and boundary value problem, Zhong \cite{Zhong5} obtained global solution under some smallness conditions on the initial data, while Zhou \cite{Zhou} established a Serrin-type blowup criterion involving only the velocity field. Later, Zhu and Ou \cite{ZhuO} extended the corresponding result in \cite{Zhong5} to the case of density-temperature-dependent viscosity. However, the global well-posedness of system \eqref{MHD} in the unbounded domains is still unknown. In fact, this is the main purpose of this paper. 

Before stating our main results, we first explain the notations and conventions used throughout this paper. We denote by 
$$\int \cdot dx = \int_{\mathbb{R}^3} \cdot dx.$$
And for $1 \le r \le \infty$ and $ k \in \mathbb{N}$, the homogeneous and inhomogeneous Sobolev spaces are defined in a standard way,

\begin{equation*}
\left\{
\begin{aligned}
& L^r = L^r (\mathbb{R}^3), \quad W^{k,r} = W^{k, r}(\mathbb{R}^3), \quad H^k = W^{k,2}, \\
& D^{k,r} = \{ f \in L^1_{\text{loc}} : \|\nabla^k f\|_{L^r} <\infty\}, \quad D^k = D^{k,2}\\
& D^1_0 = \{f \in L^6(\mathbb{R}^3): \|\nabla u\|_{L^2} < \infty\}.
\end{aligned}
\right.
\end{equation*}

Now we give the definition of strong solutions to the Cauchy problem \eqref{MHD}-\eqref{boundary} as follows.

\begin{definition}[Strong solutions] \label{def}
A pair of functions $(\rho \ge 0, \mathbf{u}, \mathbf{H}, \theta \ge 0)$ is called a strong solution to the Cauchy problem \eqref{MHD}-\eqref{boundary} in $\mathbb{R}^3 \times (0,T)$, if for some $q_0 \in (3, \infty)$, 
\begin{equation} \label{regularity}
\left\{
\begin{aligned}
&\rho \in C([0,T];H^1 \cap W^{1,q_0}), \quad  (\mathbf{u}, \mathbf{H},\theta) \in C([0,T]; D^1_0 \cap D^2) \cap L^2(0,T; D^{2,q_0}), \\
&\rho_t \in C([0,T];L^{q_0}), \quad  (\sqrt{\rho}\mathbf{u}_t, \mathbf{H}_t, \sqrt{\rho}\theta_t) \in L^{\infty}(0, T; L^2), \quad (\mathbf{u}_t, \mathbf{H}_t, \theta_t) \in L^2(0,T; D^1_0),\\
\end{aligned}
\right.
\end{equation}
and $(\rho, \mathbf{u}, \mathbf{H}, \theta)$ satisfies  \eqref{MHD} almost everywhere in $\mathbb{R}^3 \times (0,T)$.
\end{definition}

Our main results read as follows:
\begin{theorem} \label{thm1}
For constant $q \in (3, 6]$, assume that the initial data $(\rho_0 \ge 0, \mathbf{u}_0, \mathbf{H}_0, \theta_0 \ge 0)$ satisfies
\begin{equation} \label{RC}
\rho \in L^1 \cap H^1 \cap W^{1, q},~~(\mathbf{u}_0, \mathbf{H}_0, \theta_0) \in D_0^1 \cap D_2, ~~\mathrm{div}\mathbf{u}_0 = \mathrm{div}\mathbf{H}_0 = 0,
\end{equation}
and the compatibility conditions
\begin{equation} \label{CC1}
-\mathrm{div}(2\mu\mathfrak{D}(\mathbf{u}_0)) -\mathbf{H}_0 \cdot \nabla \mathbf{H}_0 + \nabla P_0= \sqrt{\rho_0} \mathbf{g}_1,
\end{equation}
and
\begin{equation} \label{CC2}
\kappa \Delta \theta_0 + 2\mu |\mathfrak{D}(\mathbf{u}_0)|^2 + \nu|\nabla \times \mathbf{H}_0|^2 = \sqrt{\rho_0} \mathbf{g}_2,
\end{equation}
for some $P_0 \in D^1$, and $\mathbf{g}_1, \mathbf{g}_2 \in L^2.$ Let $(\rho, \mathbf{u}, \mathbf{H}, \theta)$ be a strong solution in $\mathbb{R}^3 \times (0, T^*)$ as described in  Definition \ref{def}. If $T^* < \infty$ is the maximal existence time, then
\begin{equation} \label{bloweq}
\lim\limits_{T \to T^*} \|\mathbf{u}\|_{L^s(0, T; L^r_{\omega})} = \infty
\end{equation}
for any $r$ and $s$ satisfying
\begin{equation} \label{index}
\frac{2}{s} + \frac{3}{r} \le 1, \quad 3 < r \le \infty,
\end{equation}
where $L^r_{\omega}$ denotes the weak-$L^r$ space.
\end{theorem}

\begin{remark}
The local existence of unique strong solution to \eqref{MHD}-\eqref{boundary} with the initial data described in Theorem \ref{thm1} can be established in a similar way as \cite{ChoK} (see also \cite{Wu}). Hence, the maximal time $T^*$ is well-defined.
\end{remark}

\begin{remark}
It should be pointed out that the blowup criterion \eqref{bloweq} is independent of both the temperature and magnetic field, which is the same as the weak Serrin-type blowup criterion of homogeneous Navier-Stokes equations (see the work of H. Sohr \cite{Sohr}).
\end{remark}

\begin{remark}
The approach can be adapted to deal with the case of bounded domain in $\mathbb{R}^3$. And compared with \cite{Zhou} for bounded domain, some new difficulties occur in our analysis. First, the Poincare's inequality fails for 3D Cauchy problem, which is key to estimate $\|\theta\|_{L^2}.$ Furthermore, due to $\|\mathbf{u}\|_{L^s(0, T; L^r_{\omega})} \le \|\mathbf{u}\|_{L^s(0, T; L^r)}$, which implies the blowup criterion \eqref{bloweq} is stronger than that of \cite{Zhou}.
\end{remark}

The proof of Theorem \ref{thm1} will done by the contradiction argument. In view of the local existence result, to prove Theorem \ref{thm1}, it suffices to verify that $(\rho, \mathbf{u}, \mathbf{H}, \theta)$ satisfy \eqref{RC}-\eqref{CC2} at the time $T^*$ under the assumption that the left hand side of \eqref{bloweq} is finite, then apply the local existence result to extend a local solution beyond the maximal existence time $T^*$, consequently leading to a contradiction.

Based on Theorem \ref{thm1}, we can establish the global existence of strong solutions to \eqref{MHD}-\eqref{boundary} under some smallness condition on the initial data.

\begin{theorem} \label{thm2}
Let the conditions in Theorem \ref{thm1} be in force. Then there exists a small positive constant $\varepsilon_0$ depending only on $\mu, \nu$, and $\|\rho_0\|_{L^{\infty}}$ such that if 
\begin{equation} \label{small}
(\|\sqrt{\rho_0}
\mathbf{u}_0\|_{L^2}^2 + \|\mathbf{H}_0\|_{L^2}^2)(\|\nabla \mathbf{u}_0\|_{L^2}^2 + \|\nabla \mathbf{H}_0\|_{L^2}^2) \le \varepsilon_0,
\end{equation}
then the Cauchy problem of the system \eqref{MHD}-\eqref{boundary} admits a unique global strong solution.
\end{theorem}

We now comment on the analysis of this paper. The study of weak Serrin-type blowup criterion (Theorem \ref{thm1}) is mainly motivated by a recent work of Y. Wang \cite{WangY}, which established a Serrin's blowup criterion for nonhomogeneous heat conducting Navier-Stokes equations in the whole space $\mathbb{R}^3$ using the weak Lebesgue spaces. Compared to Navier-Stokes model in \cite{WangY}, the mathematical analysis of nonhomogeneous heat conducting MHD system will be more complicated on the account of the coupling of the velocity and magnetic field (such as the term $\mathbf{u}\cdot \nabla \mathbf{H}$) and strong nonlinearity (such as the term $\mathbf{H} \cdot \nabla \mathbf{H}$). And to overcome these difficulties, one of the key ideas is to derive a estimate of $\|\mathbf{H}\|_{L^{\infty}(0, T; L^q)}$ for $q>2$ which turns out to play an important role in our analysis. And it should be noted our blowup criterion is independent of the temperature and magnetic field, which means the temperature and magnetic field do not play a particular role when the singularity of solution $(\rho, \mathbf{u}, \mathbf{H}, \theta)$ forms in finite time.

As an immediate application of the blowup criterion obtained in Theorem \ref{thm1}, we plan to extend the local strong solution to be a global one under some smallness condition on the initial data. Noticing that the pair $(s, r) = (4, 6)$ satisfies $\frac{2}{s} + \frac{3}{r} \le 1,$ we conclude that the global existence of a unique strong solution can be verified if we can obtain the uniformly time independent estimate on the $L^2(0, T; L^2)$-norm of the gradient of the velocity. To this end, we multiply the momentum equations by $\mathbf{u}_t$ and make good use of the smallness of initial data to obtain the desired estimate.

The remainder of this paper is arranged as follows. In Section 2, we give some auxiliary lemmas which are useful in our later analysis. The proof of Theorem \ref{thm1} will be done by combining the contradiction argument with the estimates derived in Section 3. Finally, we give the proof of Theorem \ref{thm2} in Section 4.

\section{Preliminaries}

In this section, we will recall some known facts and analytic inequalities that will be used in the later analysis.

We begin with the following local existence and uniqueness of strong solutions when the initial data is allowed vacuum, which can be proved in a similar way as \cite{ChoK} (see also \cite{Wu}).
\begin{lemma} \label{local}
Assume that the initial data $(\rho_0, \mathbf{u}_0, \mathbf{H}_0, \theta_0)$ satisfies \eqref{RC}-\eqref{CC2}. Then there exist a positive time $T_1$ and a unique strong solution to the Cauchy problem \eqref{MHD}-\eqref{boundary} on $\mathbb{R}^3 \times (0, T_1].$
\end{lemma}

Next, we will introduce the well-known Gagliardo-Nirenberg inequality which will be used later frequently. See \cite{Nirenberg} for the proof and more details.

\begin{lemma} \label{GNin}
For $p \in [2, 6], q \in (1, +\infty),$ and $r \in (3, +\infty),$ there exists some generic constant $C$ which may depend only on $p, q$ and $r,$ such that for $f \in H^1, g \in L^q \cap D^{1, r},$ the following inequalities hold.
\begin{equation} \label{GN1}
\| f \|_{L^p} \le C \|f \|_{L^2}^{\frac{6-p}{p}} \|\nabla f\|_{L^2}^{\frac{3p-6}{2p}},
\end{equation} 
and 
\begin{equation} \label{GN2}
\| g\|_{L^{\infty}} \le C \|g\|_{L^q} + C  \|\nabla g\|_{L^r}.
\end{equation} 
\end{lemma}

Since our blowup criterion \eqref{bloweq} involves a weak Lebesgue space, it is necessary to give a short introduction and state related inequalities. Denote the Lorentz space and its norm by $L^{p,q}$ and $\|\cdot\|_{L^{p,q}}$, respectively, where $1<p< \infty$ and $ 1 \le q \le \infty$. And we recall the weak-$L^p$ space $L^p_{\omega}$ which is defined as follows:
$$ L^p_{\omega} := \{ f \in L^1_{\text{loc}}: \|f \|_{L^p_{\omega}} = \sup_{\lambda>0} \lambda|\{|f(x)| > \lambda\}|^{\frac{1}{p}} < \infty \}.$$
And it should be noted that
$$ L^p \subsetneqq L^p_{\omega}, \quad L^{\infty}_{\omega} = L^{\infty}, \quad L^p_{\omega} = L^{p,\infty}, \quad L^{p, p} = L^p, \quad \|f\|_{L^p_{\omega}} \le \|f\|_{L^p}.$$
For the details of Lorentz space, we refer to the monograph by Grafakos {\cite{Gra}}. In particular, we introduce the following H\"older inequality in Lorentz space whose proof can be found in \cite{kozono}.
\begin{lemma} \label{kozono}
Let $p_1, p_2 \in (0, \infty), q_1, q_2 \in [1, \infty]$ satisfying $\frac{1}{p} = \frac{1}{p_1} + \frac{1}{p_2} <1$ and $q= \min\{q_1, q_2\}.$ Then for $f\in L^{p_1,q_1}$ and $g \in L^{p_2, q_2},$ there exists a positive constant $C$ depending on $p_1, p_2, q_1$ and $q_2$ such that $f\cdot g \in L^{p, q}$ satisfying
\begin{equation} \label{kozono1}
\|f\cdot g\|_{L^{p,q}} \le C \|f\|_{L^{p_1, q_1}} \|g\|_{L^{p_2, q_2}}.
\end{equation}
\end{lemma}
Based on the above lemma \ref{kozono}, we have the following result involving the weak Lebesgue spaces, which will play an important role in the subsequent analysis. 
\begin{lemma} \label{weaksobolev}
Assume $g \in H^1$, and $f \in L^r_{\omega}$ with $r\in (3, \infty]$, then $f\cdot g \in L^2$. Furthermore, for any $\varepsilon >0$, we have
\begin{equation}  \label{weakembed}
\|f\cdot g\|_{L^2}^2 \le \varepsilon \|\nabla g\|_{L^2}^2 + C(\varepsilon) (\|f\|_{L^r_{\omega}}^s +1) \|g\|_{L^2}^2,
\end{equation}
where $C$ is a positive constant depending only on $\varepsilon$ and $r$.
\end{lemma}
\begin{proof} Modifying the proof in \cite{K} for bounded domains slightly, it follows from \eqref{kozono1} and interpolation inequality that
\begin{equation} 
\begin{aligned}
\|f\cdot g\|_{L^2}^2 & = \|f\cdot g\|_{L^{2,2}}^2  \le C \|f\|_{L^{r, \infty}}\|g\|_{L^{\frac{2r}{r-2}, 2}} \\
& \le  C \|f\|_{L^r_{\omega}} \|g\|_{L^{\frac{2r_1}{r_1-2}}} \|g\|_{L^{\frac{2r_2}{r_2-2}}}\\
& \le C \|f\|_{L^r_{\omega}} \|g\|_{L^2}^{\frac{r_1-2}{r_1}}\|g\|_{L^6}^{\frac{3}{r_1}} \|g\|_{L^2}^{\frac{r_2-2}{r_2}}\|g\|_{L^6}^{\frac{3}{r_2}}\\
& \le C \|f\|_{L^r_{\omega}} \|g\|_{L^2}^{\frac{2r-6}{r}}\|g\|_{L^6}^{\frac{6}{r}}\\
& \le C \|f\|_{L^r_{\omega}} \|g\|_{L^2}^{\frac{2r-6}{r}}\|\nabla g\|_{L^2}^{\frac{6}{r}}\\
& \le \varepsilon \|\nabla g\|_{L^2}^2 + C(\varepsilon)(\|f\|_{L^r_{\omega}}^s +1 ) \|g\|_{L^2}^2,
\end{aligned}
\end{equation}
where $r_1, r_2$ and $r$ satisfy $3 < r_1 < r < r_2 < \infty, \frac{2}{r}=\frac{1}{r_1} + \frac{1}{r_2}$ and $\frac{2}{s} + \frac{3}{r} \le 1.$
This completes the proof of Lemma \ref{weaksobolev}.
\end{proof}

Finally, we give classical regularity results for the Stokes system in the whole space $\mathbb{R}^3$, which have been proved in \cite{HeLi}.

\begin{lemma} \label{stokeseq}
For any $r \in (1, +\infty),$ if $\mathbf{F} \in L^r,$ there exists some positive constant $C$ depending only on $r$ such that the unique weak solution $(\mathbf{U}, P) \in D^1 \times L^2$ to the following Stokes system
\begin{equation} \label{stokes}
\left\{
\begin{aligned}
-\Delta \mathbf{U} + \nabla P = \mathbf{F},  \quad \quad &\text{in} \quad \mathbb{R}^3, \\
\mathrm{div} \mathbf{U} =0, \quad \quad &\text{in} \quad \mathbb{R}^3, \\
 \mathbf{U}(x) \rightarrow 0, \quad \quad & \text{as} \quad |x| \rightarrow + \infty, \\
\end{aligned}
\right.
\end{equation}
satisfying
\begin{equation} \label{stokesin}
\|\nabla^2 \mathbf{U}\|_{L^r} + \|\nabla P\|_{L^r} \le C \|\mathbf{F}\|_{L^r}.
\end{equation}
\end{lemma}

\section{Proof of Theorem \ref{thm1}}

This section is devoted to giving a proof of Theorem \ref{thm1} by using the contradiction argument. To do this, let $(\rho, \mathbf{u}, \mathbf{H}, \theta)$  be a strong solution to the Cauchy problem  \eqref{MHD}-\eqref{boundary} as described in Lemma \ref{local}. Suppose that \eqref{bloweq} in Theorem \ref{thm1} were false, that is to say, there exists a positive constant $M_0$ such that
\begin{equation} \label{blowup}
\lim_{T \rightarrow T^*} \| \mathbf{u} \|_{ L^s (0,T; L^r_{\omega})} \le M_0  < + \infty.
\end{equation}
Under the condition \eqref{blowup}, we will extend the existence time of the strong solution beyond $T^*$, which contradicts the definition of maximum of $T^*$. 

Before proceeding, it is easy to rewrite the system \eqref{MHD} in the following form if we assume the solution $(\rho, \mathbf{u}, \mathbf{H}, \theta)$ is regular enough:
\begin{equation} \label{MHD_1}
\left\{
\begin{aligned}
&\partial_t\rho + \mathbf{u}\cdot \nabla \rho = 0, \\
&\rho \partial_t \mathbf{u} + \rho \mathbf{u}\cdot \nabla \mathbf{u} - \mu\Delta \mathbf{u} +\nabla P = (\mathbf{H} \cdot \nabla) \mathbf{H}, \\
& c_v (\rho \partial_t \theta + \rho \mathbf{u} \cdot \nabla \theta)-\kappa \Delta \theta = 2\mu |\mathfrak{D} (\mathbf{u})|^2 + \nu |\nabla \times \mathbf{H}|^2, \\
& \partial_t \mathbf{H} + (\mathbf{u} \cdot \nabla) \mathbf{H} - (\mathbf{H}\cdot \nabla) \mathbf{u} = \nu \Delta \mathbf{H}, \\
&\mathrm{div} \mathbf{u}=\mathrm{div} \mathbf{H}= 0.
\end{aligned}
\right.
\end{equation}

In this section, the symbol $C$ denotes a generic constant which may depend on $M_0, \mu, \nu, c_v, \kappa, T^*$, and the initial data.

Now we establish some a priori estimates which will be used to prove Theorem \ref{thm1} in the end of this section.

\subsection{Lower-order estimates}

A series of key lower-order estimates to $(\rho, \mathbf{u}, \mathbf{H}, \theta)$ will be derived in this subsection.

First, it follows from the transport equation $\eqref{MHD_1}_1$ for the density and incompressible condition $\mathrm{div} \mathbf{u} =0$ that the following result holds.
\begin{lemma} \label{lem1}
There exists a positive constant $C$ satisfying
\begin{equation} \label{rho_l^p}
\sup_{0 \le t \le T} \|\rho \|_{L^1 \cap L^{\infty}} \le C, \quad \quad 0 \le T< T^*.
\end{equation}
\end{lemma}

Next, the standard energy estimates read as follows.
\begin{lemma} \label{lem2}
It holds that for any $0  \le T < T^*$, 
\begin{equation} \label{elementary}
\begin{aligned}
\sup_{0\le t\le T}\left(\|\sqrt{\rho}\mathbf{u}\|_{L^2}^2 + \|\mathbf{H}\|_{L^2}^2 + \|\rho \theta \|_{L^1} \right)+  \int_0^T (\|\nabla \mathbf{u}\|_{L^2}^2 + \|\nabla \mathbf{H}\|_{L^2}^2)dt \le C.
\end{aligned}
\end{equation}
\end{lemma}

\begin{proof}
It follows from the standard maximum principle to $\eqref{MHD}_3$ together with $\theta_0 \ge 0$ that (see \cite{Feireisl} for the proof)
$$\inf_{\mathbb{R}^3 \times [0, T]} \theta(x,t) \ge 0.$$

Moreover, multiplying $\eqref{MHD}_2$ by $\mathbf{u}$, $\eqref{MHD_1}_4$ by $\mathbf{H}$, and integrating the resulting equations over $\mathbb{R}^3$, respectively, it follows from integrating by parts that
\begin{equation} \label{energy_diff}
\frac{1}{2} \frac{d}{dt}\int (\rho|\mathbf{u}|^2 + |\mathbf{H}|^2) dx + \int (\mu |\nabla \mathbf{u}|^2 + \nu |\nabla \mathbf{H}|^2)dx =0.
\end{equation}
Integrating $\eqref{MHD}_3$ with respect to the spatial variable over $\mathbb{R}^3$ and performing integration by parts, we obtain that
\begin{equation} \label{rhotheta}
\begin{aligned}
c_v \frac{d}{dt} \int \rho\theta dx & = \int (2\mu |\mathfrak{D}(\mathbf{u})|^2 + \nu |\nabla \times \mathbf{H}|^2) dx.
\end{aligned}
\end{equation}
By the definition of $\mathfrak{D}(\mathbf{u})$ and integration by parts, we get
\begin{equation} \label{11}
\begin{aligned}
2\mu \int |\mathfrak{D}(\mathbf{u})|^2 dx & = \frac{\mu}{2} \int (\partial_i u^j + \partial_j u^i)^2 dx \\
& = \mu \int |\partial_i u^j|^2 dx + \mu \int \partial_i u^j \partial_j u^i dx \\
& = \mu \int |\nabla \mathbf{u}|^2 dx,
\end{aligned}
\end{equation}
and it follows from $-\Delta \mathbf{H}= \nabla \times (\nabla \times \mathbf{H})$ (since $\mathrm{div} \mathbf{H} =0)$ that
\begin{equation} \label{12}
\|\nabla \mathbf{H}\|_{L^2}^2 = \|\nabla \times \mathbf{H}\|_{L^2}^2.
\end{equation}
Substituting \eqref{11} and \eqref{12} into \eqref{rhotheta}, it gives
\begin{equation} \label{rhotheta_1}
\begin{aligned}
c_v \frac{d}{dt} \int \rho\theta dx = \int (\mu |\nabla \mathbf{u}|^2 + \nu |\nabla \mathbf{H}|^2) dx.
\end{aligned}
\end{equation}

Therefore, adding \eqref{energy_diff} multiplied by $2$ to \eqref{rhotheta_1}, we have
\begin{equation} \label{energy_diff1}
\frac{d}{dt}\int (\rho|\mathbf{u}|^2 + |\mathbf{H}|^2 + c_v \rho \theta) dx + \int (\mu |\nabla \mathbf{u}|^2 + \nu |\nabla \mathbf{H}|^2)dx =0.
\end{equation}
Integrating \eqref{energy_diff1} with respect to $t$ over $[0, T]$ leads to the desired \eqref{elementary}, which completes the proof of Lemma \ref{lem2}.
\end{proof}

Before deriving the key estimates of $\|\nabla \mathbf{u}\|_{L^{\infty}(0, T; L^2)}$ and $\|\nabla \mathbf{H}\|_{L^{\infty}(0, T; L^2)}$, we insert an important estimate on magnetic field $\mathbf{H}$ initiated by He and Xin \cite{HeX}, which will be stated in the following lemma.

\begin{lemma} \label{hx}
Under the condition \eqref{blowup}, it holds that for $q \in [2, 12]$ and $0 \le  T < T^*,$
\begin{equation} \label{hexin}
\sup_{0 \le t \le T} \|\mathbf{H}\|_{L^q}^q + \int_0^T \int |\mathbf{H}|^{q-2}|\nabla \mathbf{H}|^2 dx dt \le C.
\end{equation}
\end{lemma}

\begin{proof}
Multiplying $\eqref{MHD_1}_4$ by $q|\mathbf{H}|^{q-2}\mathbf{H}$ and integrating the resulting equation over $\mathbb{R}^3$, it follows from \eqref{weakembed} in Lemma \ref{weaksobolev} that
\begin{equation} \label{H1}
\begin{aligned}
\frac{d}{d t} & \int|\mathbf{H}|^{q} d x+\nu \int\left(q|\mathbf{H}|^{q-2}|\nabla \mathbf{H}|^{2}+ q(q-2)|\mathbf{H}|^{q-2}\left|\nabla |\mathbf{H}| \right|^2 \right) d x \\
=&-\int q|\mathbf{H}|^{q-2}\left(\mathbf{H}\cdot \nabla \mathbf{H} \cdot \mathbf{u}-\frac{q-1}{2} \mathbf{u} \cdot \nabla|\mathbf{H}|^{2}\right) d x \\
&-\frac{q(q-2)}{2} \int|\mathbf{H}|^{q-4}\left(\mathbf{H} \cdot \nabla|\mathbf{H}|^{2}\right)(\mathbf{u} \cdot \mathbf{H}) d x \\
\leq & \frac{\nu}{2} \int q|\mathbf{H}|^{q-2}|\nabla \mathbf{H}|^{2} d x+C q^{2} \int|\mathbf{u}|^{2}|\mathbf{H}|^{q} d x \\
= & \frac{\nu}{2} \int q|\mathbf{H}|^{q-2}|\nabla \mathbf{H}|^{2} d x+C q^{2} \left\||\mathbf{u}||\mathbf{H}|^{\frac{q}{2}}\right\|_{L^2}^2 \\
\leq & \frac{\nu}{2} \int q|\mathbf{H}|^{q-2}|\nabla \mathbf{H}|^{2} d x+ \varepsilon \|\nabla |\mathbf{H}|^{\frac{q}{2}} \|_{L^2}^2 + C(\varepsilon)(1+\|\mathbf{u}\|_{L^r_{\omega}}^s) \|\mathbf{H}\|_{L^q}^q. 
\end{aligned}
\end{equation}
Choosing $\varepsilon$ suitably small in \eqref{H1}, we obtain the desired \eqref{hexin} after applying Gronwall's inequality and \eqref{blowup}. Therefore the proof of Lemma \ref{hx} is completed.
\end{proof}

With the help of Lemma \ref{hx}, we can now derive key time-independent estimates on the $L^{\infty}(0, T; L^{2})$-norm of the gradients of velocity and magnetic field.

\begin{lemma} \label{nablauH}
Under the assumption \eqref{blowup}, it holds for all $0 \le  T < T^*$,
\begin{equation} \label{nablauH_1}
\sup_{0 \le t \le T} (\|\nabla \mathbf{u}\|_{L^2}^2 + \|\nabla \mathbf{H}\|_{L^2}^2) + \int_0^T \left(\|\sqrt{\rho}\mathbf{u}_t\|_{L^2}^2 + \|\mathbf{H}_t\|_{L^2}^2 + \|\nabla^2 \mathbf{u}\|_{L^2}^2 + \|\nabla^2 \mathbf{H}\|_{L^2}^2\right) dt \le C.
\end{equation}
\end{lemma}

\begin{proof}
Multiplying $\eqref{MHD_1}_2$ by $\mathbf{u}_t$ and integrating the resulting equation over $\mathbb{R}^3$ lead to
\begin{equation} \label{nablau}
\frac{\mu}{2}\frac{d}{dt}\int |\nabla \mathbf{u}|^2 dx + \int \rho |\mathbf{u}_t|^2 dx = -\int \rho \mathbf{u}\cdot \nabla \mathbf{u}\cdot \mathbf{u}_t dx + \int \mathbf{H} \cdot \nabla \mathbf{H} \cdot \mathbf{u}_t dx.
\end{equation}
And it follows from the equation $\eqref{MHD_1}_4$ that $\mathbf{H}_t-\nu \Delta \mathbf{H} = \mathbf{H}\cdot \nabla \mathbf{u}-\mathbf{u}\cdot \nabla \mathbf{H},$ then we have
\begin{equation} \label{magneto}
\int |\mathbf{H}_t-\nu \Delta \mathbf{H}|^2 dx = \int |\mathbf{u}\cdot \nabla \mathbf{H}- \mathbf{H}\cdot \nabla \mathbf{u}|^2 dx.
\end{equation}
For the L. H. S. of \eqref{magneto}, it is easy to get
\begin{equation} \label{magneto_1}
\begin{aligned}
\int |\mathbf{H}_t-\nu \Delta \mathbf{H}|^2 dx & = \int (|\mathbf{H}_t|^2 + \nu^2 |\Delta \mathbf{H}|^2 - 2\nu \mathbf{H}_t \cdot \Delta \mathbf{H}) dx \\
& = \nu \frac{d}{dt} \int |\nabla \mathbf{H}|^2 dx + \int (|\mathbf{H}_t|^2 + \nu^2 |\Delta \mathbf{H}|^2) dx.
\end{aligned}
\end{equation}
Substituting \eqref{magneto_1} into the L. H. S. of \eqref{magneto}, we obtain that
\begin{equation} \label{nablaH}
\nu \frac{d}{dt} \int |\nabla \mathbf{H}|^2 dx + \int (|\mathbf{H}_t|^2 + \nu^2 |\Delta \mathbf{H}|^2) dx = \int |\mathbf{u}\cdot \nabla \mathbf{H}- \mathbf{H}\cdot \nabla \mathbf{u}|^2 dx.
\end{equation}

Notice that the standard $L^2$-estimate of elliptic system gives 
\begin{equation} \label{elli}
\|\nabla^2 \mathbf{H} \|_{L^2}^2 \le K \|\Delta \mathbf{H}\|_{L^2}^2,
\end{equation}
with a positive constant $K$.

Adding \eqref{nablau} to \eqref{nablaH}, we derive from Cauchy-Schwarz inequality and \eqref{elli} that
\begin{equation}
\begin{aligned}
& \frac{d}{dt} \left(\frac{\mu}{2} |\nabla \mathbf{u}|^2 + \nu |\nabla \mathbf{H}|^2 \right) dx + \int \left(\rho |\mathbf{u}_t|^2 + |\mathbf{H}_t|^2 + \frac{\nu^2}{K} |\nabla^2 \mathbf{H}|^2 \right) dx \\
= & \int \mathbf{H} \cdot \nabla \mathbf{H} \cdot \mathbf{u}_t dx -\int \rho \mathbf{u} \cdot \nabla \mathbf{u} \cdot \mathbf{u}_t dx  + \int |\mathbf{u}\cdot \nabla \mathbf{H}- \mathbf{H}\cdot \nabla \mathbf{u}|^2 dx \\
= & - \frac{d}{dt} \int (\mathbf{H}\cdot \nabla) \mathbf{u}\cdot \mathbf{H} dx + \int (\mathbf{H}_t \cdot \nabla) \mathbf{u}\cdot \mathbf{H}_t dx + \int (\mathbf{H} \cdot \nabla) \mathbf{u} \cdot \mathbf{H}_t dx \\
& -\int \rho \mathbf{u}\cdot \nabla \mathbf{u} \cdot \mathbf{u}_t dx + \int |\mathbf{u}\cdot \nabla \mathbf{H}- \mathbf{H}\cdot \nabla \mathbf{u}|^2 dx \\
\le & -\frac{d}{dt} \int (\mathbf{H}\cdot \nabla) \mathbf{u} \cdot \mathbf{H} dx +\frac{1}{2} \int \rho|\mathbf{u}_t|^2  dx + \frac{1}{2} \int |\mathbf{H}_t|^2 dx \\
& + C \int |\sqrt{\rho} \mathbf{u} \cdot \nabla \mathbf{u}|^2 dx + C \int |\mathbf{u} \cdot \nabla \mathbf{H}|^2 dx + C \int |\mathbf{H} \cdot \nabla \mathbf{u}|^2 dx.
\end{aligned}
\end{equation}
Thus, we obtain that
\begin{equation} \label{energy2u}
\begin{aligned}
& \frac{d}{dt} \int \left(\mu |\nabla \mathbf{u}|^2 + 2\nu |\nabla \mathbf{H}|^2 + 2(\mathbf{H}\cdot \nabla)\mathbf{u}\cdot \mathbf{H}\right) dx + \int \left(\rho |\mathbf{u}_t|^2 + |\mathbf{H}_t|^2 + \frac{2\nu^2}{K} |\nabla^2 \mathbf{H}|^2 \right) dx \\
\le & C \int |\sqrt{\rho} \mathbf{u}\cdot \nabla \mathbf{u}|^2 dx + C \int |\mathbf{u} \cdot \nabla \mathbf{H}|^2 dx + C \int |\mathbf{H} \cdot \nabla \mathbf{u}|^2 dx.
\end{aligned}
\end{equation}
Recall that $(\mathbf{u}, P)$ satisfies the following Stokes system
\begin{equation}
\begin{cases}
-\mu \Delta \mathbf{u}+\nabla P=-\rho \mathbf{u}_t -\rho \mathbf{u}\cdot \nabla \mathbf{u} + \mathbf{H}\cdot \nabla \mathbf{H}, & x \in \mathbb{R}^3, \\ \operatorname{div} \mathbf{u} =0, & x \in \mathbb{R}^3, \\ \mathbf{u}(x) \rightarrow 0,  & |x| \rightarrow \infty. 
\end{cases}
\end{equation}
Applying Lemma \ref{stokeseq} with $\mathbf{F} \triangleq -\rho \mathbf{u}_t -\rho \mathbf{u}\cdot \nabla \mathbf{u} + \mathbf{H}\cdot \nabla \mathbf{H}$, we obtain from \eqref{stokesin} that
\begin{equation} \label{nabla2u}
\begin{aligned}
\|\nabla^2 \mathbf{u}\|_{L^2}^2 & \le C(\|\rho \mathbf{u}_t\|_{L^2}^2 + \|\rho \mathbf{u}\cdot \nabla \mathbf{u}\|_{L^2}^2 + \|\mathbf{H}\cdot \nabla \mathbf{H}\|_{L^2}^2) \\
& \le L(\|\sqrt{\rho} \mathbf{u}_t\|_{L^2}^2 + \|\sqrt{\rho} \mathbf{u}\cdot \nabla \mathbf{u}\|_{L^2}^2 + \|\mathbf{H}\cdot \nabla \mathbf{H}\|_{L^2}^2),
\end{aligned}
\end{equation}
where $L$ is a positive constant depending only on $\mu$ and $\|\rho_0\|_{L^{\infty}}.$ 

Adding \eqref{nabla2u} multiplied by $\frac{1}{2L}$ to \eqref{energy2u}, we have
\begin{equation} \label{energy3u}
\begin{aligned}
& \frac{d}{dt} \int \left(\mu |\nabla \mathbf{u}|^2 + 2\nu |\nabla \mathbf{H}|^2 + 2(\mathbf{H}\cdot \nabla)\mathbf{u}\cdot \mathbf{H} \right) dx  \\
& + \int \left(\frac{1}{2}\rho |\mathbf{u}_t|^2 + |\mathbf{H}_t|^2 + \frac{2\nu^2}{K} |\nabla^2 \mathbf{H}|^2 + \frac{1}{2L}\|\nabla^2 \mathbf{u}\|_{L^2}^2 \right) dx \\
\le  & C \int \left( |\sqrt{\rho} \mathbf{u} \cdot \nabla \mathbf{u}|^2  +  |\mathbf{u} \cdot \nabla \mathbf{H}|^2  + |\mathbf{H} \cdot \nabla \mathbf{u}|^2 +  |\mathbf{H} \cdot \nabla \mathbf{H}|^2 \right)dx \\
\le & C \|\rho\|_{L^{\infty}}\|\mathbf{u} \cdot \nabla \mathbf{u}\|_{L^2}^2 + C \|\mathbf{u} \cdot \nabla \mathbf{H}\|_{L^2}^2 \\
& + C \|\mathbf{H}\|_{L^6}^2 \|\nabla \mathbf{u}\|_{L^2} \|\nabla \mathbf{u}\|_{L^6} + C \|\mathbf{H}\|_{L^6}^2 \|\nabla \mathbf{H}\|_{L^2} \|\nabla \mathbf{H}\|_{L^6} \\
\le & \frac{\varepsilon}{2} (\|\nabla^2 \mathbf{u}\|_{L^2}^2 + \|\nabla^2 \mathbf{H}\|_{L^2}^2) + C(\varepsilon) (1+ \|\mathbf{u}\|_{L^r_{\omega}}^s) (\|\nabla \mathbf{u}\|_{L^2}^2 + \|\nabla \mathbf{H}\|_{L^2}^2) \\
& + C \|\mathbf{H}\|_{L^6}^2\|\nabla \mathbf{u}\|_{L^2} \|\nabla^2 \mathbf{u}\|_{L^2} + C \|\mathbf{H}\|_{L^6}^2\|\nabla \mathbf{H}\|_{L^2} \|\nabla^2 \mathbf{H}\|_{L^2} \\
\le & \varepsilon(\|\nabla^2 \mathbf{u}\|_{L^2}^2 + \|\nabla^2 \mathbf{H}\|_{L^2}^2) + C(\varepsilon) (1+ \|\mathbf{u}\|_{L^r_{\omega}}^s) (\|\nabla \mathbf{u}\|_{L^2}^2 + \|\nabla \mathbf{H}\|_{L^2}^2),
\end{aligned}
\end{equation}
due to Lemma \ref{weaksobolev}, \eqref{hexin} and Cauchy-Schwarz inequality.

Moreover, applying the Cauchy-Schwarz inequality and \eqref{hexin} gives rise to
\begin{equation}
\begin{aligned}
2 \int (\mathbf{H} \cdot \nabla)\mathbf{u} \cdot \mathbf{H} dx & \le C\|\mathbf{H}\|_{L^4}^2 \|\nabla \mathbf{u}\|_{L^2} \\
& \le \frac{\mu}{2} \|\nabla \mathbf{u}\|_{L^2}^2 + C \|\mathbf{H}\|_{L^4}^4 \\
& \le \frac{\mu}{2} \|\nabla \mathbf{u}\|_{L^2}^2 + C.
\end{aligned}
\end{equation}
Together with Gronwall's inequality, \eqref{energy3u} implies
\begin{equation}
\sup_{0 \le t \le T} (\|\nabla \mathbf{u}\|_{L^2}^2 + \|\nabla \mathbf{H}\|_{L^2}^2) + \int_0^T \left(\|\sqrt{\rho}\mathbf{u}_t\|_{L^2}^2 + \|\mathbf{H}_t\|_{L^2}^2 + \|\nabla^2 \mathbf{u}\|_{L^2}^2 + \|\nabla^2 \mathbf{H}\|_{L^2}^2\right) dt \le C.
\end{equation}
Therefore the proof of Lemma \ref{nablauH} is completed.
\end{proof}

The following lemma concerns the higher regularity of the temperature $\theta$.
\begin{lemma} \label{theta}
Under the assumption \eqref{blowup}, it holds that for $0 \le  T < T^*,$
\begin{equation}
\sup_{0 \le t \le T} \|\sqrt{\rho} \theta\|_{L^2}^2 + \int_0^T \|\nabla \theta\|_{L^2}^2 dt  \le C.
\end{equation}
\end{lemma}

\begin{proof}
Multiplying $\eqref{MHD_1}_3$ by $\theta$ and integrating the resulting equation over $\mathbb{R}^3$ yield
\begin{equation} \label{theta1}
\begin{aligned}
c_v \frac{d}{dt} \int \rho\theta^2 dx + 2\kappa \int |\nabla \theta|^2 dx \le C \int |\nabla \mathbf{u}|^2 \theta dx + C \int |\nabla \mathbf{H}|^2 \theta dx.
\end{aligned}
\end{equation}
We estimate each term of R. H. S. of \eqref{theta1} as follows. Applying H\"{o}lder and Cauchy-Schwarz inequalities gives
\begin{equation} \label{es1}
\begin{aligned}
\int |\nabla \mathbf{u}|^2 \theta dx & \le C \|\nabla \mathbf{u}\|_{L^{\frac{12}{5}}}^2 \|\theta\|_{L^6} \\
& \le C \|\nabla \mathbf{u}\|_{L^2}^{\frac{3}{2}} \|\nabla \mathbf{u}\|_{L^6}^{\frac{1}{2}}\|\nabla \theta\|_{L^2} \\
& \le \frac{\kappa}{2} \|\nabla \theta\|_{L^2}^2 + C \|\nabla \mathbf{u}\|_{L^2}^3 \|\nabla^2 \mathbf{u}\|_{L^2} \\
& \le \frac{\kappa}{2} \|\nabla \theta\|_{L^2}^2 + C \|\nabla^2 \mathbf{u}\|_{L^2}^2 + C,
\end{aligned}
\end{equation}
due to \eqref{nablauH_1}. And in a similar way, we have
\begin{equation} \label{es2}
\begin{aligned}
\int |\nabla \mathbf{H}|^2 \theta dx & \le C \|\nabla \mathbf{H}\|_{L^{\frac{12}{5}}}^2 \|\theta\|_{L^6} \\
& \le C \|\nabla \mathbf{H}\|_{L^2}^{\frac{3}{2}} \|\nabla \mathbf{H}\|_{L^6}^{\frac{1}{2}}\|\nabla \theta\|_{L^2} \\
& \le \frac{\kappa}{2} \|\nabla \theta\|_{L^2}^2 + C \|\nabla \mathbf{H}\|_{L^2}^3 \|\nabla^2 \mathbf{H}\|_{L^2} \\
& \le \frac{\kappa}{2} \|\nabla \theta\|_{L^2}^2 + C \|\nabla^2 \mathbf{H}\|_{L^2}^2 + C.
\end{aligned}
\end{equation}
Substituting \eqref{es1} and \eqref{es2} into \eqref{theta1}, we obtain
\begin{equation} \label{theta2}
\begin{aligned}
c_v \frac{d}{dt} \int \rho\theta^2 dx + \kappa \int |\nabla \theta|^2 dx \le C \|\nabla^2 \mathbf{u}\|_{L^2}^2 + C \|\nabla^2 \mathbf{H}\|_{L^2}^2 + C.
\end{aligned}
\end{equation}
Integrating the above inequality \eqref{theta2} with respect to the time variable over $(0, t)$, we get from \eqref{nablauH_1} that
\begin{equation}
\begin{aligned}
\sup_{0 \le t \le T} \|\sqrt{\rho} \theta\|_{L^2}^2 + \int_0^T \|\nabla \theta\|_{L^2}^2 dt  \le C.
\end{aligned}
\end{equation}
Therefore the proof of Lemma \ref{theta} is completed.
\end{proof}

At the end of this subsection, we give the following remark which will be used later.
\begin{remark}
In view of Lemma \ref{nablauH}, we deduce from classical $L^2$-estimates for elliptic and Stokes system that
\begin{equation} \label{u22}
\begin{aligned}
\|\nabla^2 \mathbf{u}\|_{L^2}^2 & \le C(\|\sqrt{\rho} \mathbf{u}_t\|_{L^2}^2 + \|\mathbf{u} \cdot \nabla \mathbf{u}\|_{L^2}^2 + \|\mathbf{H} \cdot \nabla \mathbf{H}\|_{L^2}^2) \\
& \le C(\|\sqrt{\rho} \mathbf{u}_t\|_{L^2}^2 + \|\mathbf{u}\|_{L^6}^2 \|\nabla \mathbf{u}\|_{L^3}^2 + \|\mathbf{H}\|_{L^6}^2  \|\nabla \mathbf{H}\|_{L^3}^2) \\
& \le C(\|\sqrt{\rho} \mathbf{u}_t\|_{L^2}^2 + \|\nabla \mathbf{u} \|_{L^2}^3 \|\nabla ^2 \mathbf{u}\|_{L^2} + \|\nabla \mathbf{H}\|_{L^2}^3 \|\nabla^2 \mathbf{H}\|_{L^2}) \\
& \le \frac{1}{4}(\|\nabla^2 \mathbf{u}\|_{L^2}^2 + \|\nabla^2 \mathbf{H}\|_{L^2}^2 ) + C\|\sqrt{\rho} \mathbf{u}_t\|_{L^2}^2 + C \|\nabla \mathbf{u} \|_{L^2}^6  + C \|\nabla \mathbf{H}\|_{L^2}^6,
\end{aligned}
\end{equation}
and 
\begin{equation} \label{H22}
\begin{aligned}
\|\nabla^2 \mathbf{H}\|_{L^2}^2 & \le C(\|\mathbf{H}_t\|_{L^2}^2 + \|\mathbf{u} \cdot \nabla \mathbf{H}\|_{L^2}^2 + \|\mathbf{H} \cdot \nabla \mathbf{u}\|_{L^2}^2) \\
& \le C(\|\mathbf{H}_t\|_{L^2}^2 + \|\mathbf{u} \|_{L^6}^2 \|\nabla \mathbf{H}\|_{L^3}^2 + \|\mathbf{H}\|_{L^6}^2  \|\nabla \mathbf{u}\|_{L^3}^2) \\
& \le C(\|\mathbf{H}_t\|_{L^2}^2 + \|\nabla \mathbf{u} \|_{L^2}^2 \|\nabla \mathbf{H}\|_{L^2} \|\nabla ^2 \mathbf{H}\|_{L^2} + \|\nabla \mathbf{H}\|_{L^2}^2 \|\nabla \mathbf{u}\|_{L^2} \|\nabla^2 \mathbf{u}\|_{L^2}) \\
& \le \frac{1}{4}(\|\nabla^2 \mathbf{u}\|_{L^2}^2 + \|\nabla^2 \mathbf{H}\|_{L^2}^2 ) + C\|\mathbf{H}_t\|_{L^2}^2 + C \|\nabla \mathbf{u} \|_{L^2}^6  + C \|\nabla \mathbf{H}\|_{L^2}^6.
\end{aligned}
\end{equation}
Adding \eqref{u22} to \eqref{H22}, it follows from \eqref{nablauH_1} that
\begin{equation} \label{H22}
\begin{aligned}
\|\nabla^2 \mathbf{u}\|_{L^2}^2 + \|\nabla^2 \mathbf{H}\|_{L^2}^2 & \le C(\|\sqrt{\rho} \mathbf{u}_t\|_{L^2}^2 + \|\mathbf{H}_t\|_{L^2}^2) + C.
\end{aligned}
\end{equation}

\end{remark}

\subsection{Higher-order estimates}

In this subsection, we will estimate a series of higher-order estimates of $(\rho, \mathbf{u}, \mathbf{H},\theta).$

Firstly, we will estimate the $L^{\infty}(0, T; L^2)$-norm of $\sqrt{\rho}\mathbf{u}_t, \mathbf{H}_t$ and $\nabla \theta.$

\begin{lemma} \label{rhotheta}
Under the assumption \eqref{blowup}, it holds that for any $0 \le T < T^*,$
\begin{equation} \label{rhothetaeq}
\begin{aligned}
\sup_{0 \le t \le T} (\|\sqrt{\rho}\mathbf{u}_t \|_{L^2}^2 + \|\mathbf{H}_t\|_{L^2}^2 + \|\nabla \theta \|_{L^2}^2) + \int_0^T (\|\nabla \mathbf{u}_t\|_{L^2}^2 + \|\nabla \mathbf{H}_t\|_{L^2}^2 + \|\sqrt{\rho}\theta_t\|_{L^2}^2)dt \le C.
\end{aligned}
\end{equation}
\end{lemma}

\begin{proof}
Differentiating $\eqref{MHD_1}_2$ with respect to $t$ yields
\begin{equation} \label{utt}
\rho \mathbf{u}_{tt} + \rho \mathbf{u}\cdot \nabla \mathbf{u}_t -\mu \Delta \mathbf{u}_t = -\nabla P_t - \rho_t \mathbf{u}_t - (\rho \mathbf{u})_t \cdot \nabla \mathbf{u} + (\mathbf{H}_t \cdot \nabla)\mathbf{H} + (\mathbf{H} \cdot \nabla)\mathbf{H}_t.
\end{equation}
Multiplying the above equality \eqref{utt} by $\mathbf{u}_t$, and integrating the resulting equation over $\mathbb{R}^3,$ one obtains
\begin{equation} \label{utteq}
\begin{aligned}
&\frac{1}{2}\frac{d}{dt}\int \rho|\mathbf{u}_t|^2 dx + \mu \int |\nabla \mathbf{u}_t |^2 dx \\
=& \int \rho_t |\mathbf{u}_t|^2 dx - \int (\rho \mathbf{u})_t \cdot \nabla \mathbf{u} \cdot \mathbf{u}_t dx \\
&+ \int (\mathbf{H}_t \cdot \nabla) \mathbf{H}\cdot \mathbf{u}_t dx + \int (\mathbf{H} \cdot \nabla) \mathbf{H}_t \cdot \mathbf{u}_t dx.
\end{aligned}
\end{equation}
Next, differentiating $\eqref{MHD_1}_4$ with respect to $t$ yields
\begin{equation} \label{Htt}
\mathbf{H}_{tt} + \mathbf{u}\cdot \nabla \mathbf{H}_t -\nu \Delta \mathbf{H}_t = - \mathbf{u}_t \cdot \nabla \mathbf{H} + (\mathbf{H}_t \cdot \nabla)\mathbf{u} + (\mathbf{H} \cdot \nabla)\mathbf{u}_t.
\end{equation}
Multiplying the above equality \eqref{Htt} by $\mathbf{H}_t$, and integrating the resulting equation over $\mathbb{R}^3,$ one obtains
\begin{equation} \label{Htteq}
\begin{aligned}
&\frac{1}{2}\frac{d}{dt}\int |\mathbf{H}_t|^2 dx + \nu \int |\nabla \mathbf{H}_t |^2 dx \\
=& - \int  \mathbf{u}_t \cdot \nabla \mathbf{H} \cdot \mathbf{H}_t dx 
+ \int (\mathbf{H}_t \cdot \nabla) \mathbf{u}\cdot \mathbf{H}_t dx + \int (\mathbf{H} \cdot \nabla) \mathbf{u}_t \cdot \mathbf{H}_t dx.
\end{aligned}
\end{equation}
Adding \eqref{Htteq} into \eqref{utteq}, we get
\begin{equation} \label{uHtt}
\begin{aligned}
&\frac{1}{2}\frac{d}{dt}\int (\rho |\mathbf{u}_t|^2 + |\mathbf{H}_t|^2) dx +  \int (\mu |\nabla \mathbf{u}_t|^2 + \mu |\nabla \mathbf{H}_t |^2 ) dx \\
=& - \int \rho_t |\mathbf{u}_t|^2 dx - \int (\rho \mathbf{u})_t \cdot \nabla \mathbf{u} \cdot \mathbf{u}_t dx \\
& + \int (\mathbf{H}_t \cdot \nabla) \mathbf{H}\cdot \mathbf{u}_t dx  - \int  \mathbf{u}_t \cdot \nabla \mathbf{H} \cdot \mathbf{H}_t dx 
+ \int (\mathbf{H}_t \cdot \nabla) \mathbf{u}\cdot \mathbf{H}_t dx \\
\triangleq & I_1 + I_2 + I_3 + I_4 + I_5.
\end{aligned}
\end{equation}
The terms $I_1$-$I_5$ are estimated as follows. Indeed, it follows from \eqref{GN1}, \eqref{GN2}, \eqref{nablauH_1}, integrating by parts and H\"older inequality that
\begin{equation} \label{i1}
\begin{aligned}
I_1 & \le \int \rho |\mathbf{u}||\mathbf{u}_t| |\nabla \mathbf{u}_t| dx \\
& \le C \|\rho\|_{L^{\infty}}^{\frac{1}{2}} \|\mathbf{u}\|_{L^6} \|\sqrt{\rho}\mathbf{u}_t \|_{L^3} \|\nabla \mathbf{u}_t\|_{L^2} \\
& \le C  \|\nabla \mathbf{u}\|_{L^2} \|\sqrt{\rho}\mathbf{u}_t \|_{L^2}^{\frac{1}{2}} \|\sqrt{\rho}\mathbf{u}_t \|_{L^6}^{\frac{1}{2}} \|\nabla \mathbf{u}_t\|_{L^2} \\
& \le C \|\sqrt{\rho}\mathbf{u}_t \|_{L^2}^{\frac{1}{2}} \|\nabla \mathbf{u}_t \|_{L^6}^{\frac{3}{2}} \\
& \le \frac{\mu}{8} \|\nabla \mathbf{u}_t\|_{L^2}^2 + C \|\sqrt{\rho}\mathbf{u}_t \|_{L^2}^2,\\
\end{aligned}
\end{equation}
and 
\begin{equation}\label{i2}
\begin{aligned}
I_2 & \le \int (\rho |\mathbf{u}||\nabla \mathbf{u}|^2 |\mathbf{u}_t| + \rho |\mathbf{u}|^2 |\nabla^2 \mathbf{u}||\mathbf{u}_t| + \rho|\mathbf{u}|^2 |\nabla \mathbf{u}||\nabla \mathbf{u}_t|) dx \\
& \le C (\|\mathbf{u}\|_{L^6}\|\nabla \mathbf{u}\|_{L^6}^2 \|\sqrt{\rho} \mathbf{u}_t\|_{L^2} + \|\mathbf{u}\|_{L^6}^2 \|\nabla^2 \mathbf{u}\|_{L^2} \|\mathbf{u}_t\|_{L^6} \\
& \quad + \|\mathbf{u}\|_{L^6}^2 \|\nabla \mathbf{u}\|_{L^6} \|\nabla \mathbf{u}_t\|_{L^2}   ) \\
& \le C (\|\nabla \mathbf{u}\|_{L^2}\|\nabla^2 \mathbf{u}\|_{L^2}^2 \|\sqrt{\rho} \mathbf{u}_t\|_{L^2} + \|\nabla \mathbf{u}\|_{L^2}^2 \|\nabla^2 \mathbf{u}\|_{L^2} \|\nabla \mathbf{u}_t\|_{L^2}  \\
&\quad + \|\nabla \mathbf{u}\|_{L^2}^2 \|\nabla^2 \mathbf{u}\|_{L^2} \|\nabla \mathbf{u}_t\|_{L^2}   ) \\
& \le \frac{\mu}{8} \|\nabla \mathbf{u}_t\|_{L^2}^2 + C \|\sqrt{\rho} \mathbf{u}_t\|_{L^2}^4 +C\|\mathbf{H}_t\|_{L^2}^4 +  C,
\end{aligned}
\end{equation}
duo to \eqref{H22}. Similarly,
\begin{equation} \label{i3}
\begin{aligned}
I_3 & \le  \|\sqrt{\rho} \mathbf{u}_t\|_{L^4}^2  \|\nabla \mathbf{u}\|_{L^2} \\
& \le C \|\sqrt{\rho} \mathbf{u}_t\|_{L^2}^{\frac{1}{2}} \|\sqrt{\rho} \mathbf{u}_t\|_{L^6}^{\frac{3}{2}}\\
&\le C \|\sqrt{\rho} \mathbf{u}_t\|_{L^2}^{\frac{1}{2}} \|\nabla \mathbf{u}_t\|_{L^2}^{\frac{3}{2}}\\
& \le \frac{\mu}{8} \|\nabla \mathbf{u}_t\|_{L^2}^2 + C \|\sqrt{\rho} \mathbf{u}_t\|_{L^2}^2,
\end{aligned}
\end{equation}
\begin{equation} \label{i4}
\begin{aligned}
I_4 & \le  C \|\nabla \mathbf{H}\|_{L^2} \|\mathbf{H}_t\|_{L^3}  \|\mathbf{u}_t\|_{L^6} \\
& \le C \|\mathbf{H}_t\|_{L^2}^{\frac{1}{2}} \|\mathbf{H}_t\|_{L^6}^{\frac{1}{2}}\|\nabla \mathbf{u}_t\|_{L^2} \\
& \le \frac{\mu}{8} \|\nabla \mathbf{u}_t\|_{L^2}^2 + C \|\mathbf{H}_t\|_{L^2} \|\nabla \mathbf{H}_t\|_{L^2} \\
& \le \frac{\mu}{8} \|\nabla \mathbf{u}_t\|_{L^2}^2 + \frac{\nu}{4}\|\nabla \mathbf{H}_t\|_{L^2}^2 +   C \|\mathbf{H}_t\|_{L^2}^2, \\
\end{aligned}
\end{equation}
and
\begin{equation} \label{i5}
\begin{aligned}
I_5 & \le  \|\nabla \mathbf{u}\|_{L^2} \|\mathbf{H}_t\|_{L^4}^2  \\
& \le C \|\mathbf{H}_t\|_{L^2}^{\frac{3}{2}} \|\mathbf{H}_t\|_{L^6}^{\frac{1}{2}} \\
& \le  C \|\mathbf{H}_t\|_{L^2}^{\frac{3}{2}} \|\nabla \mathbf{H}_t\|_{L^2}^{\frac{1}{2}} \\
& \le \frac{\nu}{4} \|\nabla \mathbf{H}_t\|_{L^2}^2 +   C \|\mathbf{H}_t\|_{L^2}^2. \\
\end{aligned}
\end{equation}
Substituting \eqref{i1}-\eqref{i5} into \eqref{uHtt}, one has
\begin{equation}
\begin{aligned}
\frac{d}{dt}\int (\rho|\mathbf{u}_t|^2 + |\mathbf{H}_t|^2) dx + (\mu \|\sqrt{\rho}\mathbf{u}_t\|_{L^2}^2 + \nu \|\mathbf{H}_t\|_{L^2}^2 ) \le C (\|\sqrt{\rho} \mathbf{u}_t\|_{L^2}^4 + \|\mathbf{H}_t\|_{L^2}^4) + C.
\end{aligned}
\end{equation}
This together with Gronwall's inequality and \eqref{nablauH_1} yields
\begin{equation} \label{uHtteq}
\sup_{0 \le t \le T} (\|\sqrt{\rho}\mathbf{u}_t\|_{L^2}^2 + \|\mathbf{H}_t\|_{L^2}^2) + \int_0^T (\|\nabla \mathbf{u}_t\|_{L^2}^2 + \|\nabla \mathbf{H}_t\|_{L^2}^2) dt \le C. 
\end{equation}
Now we will prove the boundedness of $\|\nabla \theta\|_{L^2}^2(t).$ Indeed, multiplying $\eqref{MHD_1}_3$ by $\theta_t$ and integrating the resulting equation over $\mathbb{R}^3,$ one has
\begin{equation} \label{thetat}
\begin{aligned}
&\frac{\kappa}{2}\frac{d}{dt} \int |\nabla \theta|^2 dx + c_v \int \rho \theta_t^2 dx \\
= & -c_v \int \rho \mathbf{u}\cdot \nabla \theta \cdot \theta_t dx + 2\mu \int |\mathfrak{D}(\mathbf{u})|^2 \theta_t dx + \nu \int |\nabla \times \mathbf{H}|\theta_t dx \\
\triangleq & J_1 + J_2 + J_3.
\end{aligned}
\end{equation}
In view of H\"older inequality, \eqref{nablauH_1} and \eqref{uHtteq}, we obtain
\begin{equation} \label{j1}
\begin{aligned}
J_1 & \le C \|\mathbf{u}\|_{L^{\infty}}\|\sqrt{\rho} \theta_t \|_{L^2} \|\nabla \theta\|_{L^2} \\
& \le C \|\mathbf{u}\|_{W^{1, 6}} \|\sqrt{\rho} \theta_t\|_{L^2} \|\nabla \theta\|_{L^2}\\
& \le C \|\nabla \mathbf{u}\|_{H^1} \|\sqrt{\rho} \theta_t\|_{L^2} \|\nabla \theta\|_{L^2}\\
& \le \frac{c_v}{2} \|\sqrt{\rho} \theta_t\|_{L^2}^2 + C  \|\nabla \theta\|_{L^2}^2.
\end{aligned}
\end{equation}
Furthermore, we get
\begin{equation} \label{j2}
\begin{aligned}
J_{2} &=2 \mu \frac{d}{d t} \int|\mathfrak{D}(\mathbf{u})|^{2} \theta d x-2 \mu \int\left(|\mathfrak{D}(\mathbf{u})|^{2}\right)_{t} \theta d x \\
& \leq 2 \mu \frac{d}{d t} \int|\mathfrak{D}(\mathbf{u})|^{2} \theta d x+C \int\left|\nabla \mathbf{u} \| \nabla \mathbf{u}_{t}\right| \theta d x \\
& \leq 2 \mu \frac{d}{d t} \int|\mathfrak{D}(\mathbf{u})|^{2} \theta d x+C\|\nabla \mathbf{u}\|_{L^{3}}\left\|\nabla \mathbf{u}_{t}\right\|_{L^{2}}\|\theta\|_{L^{6}} \\
& \leq 2 \mu \frac{d}{d t} \int|\mathfrak{D}(\mathbf{u})|^{2} \theta d x+C\|\nabla \mathbf{u}\|_{L^{2}}^{\frac{1}{2}}\left\|\nabla^{2} \mathbf{u}\right\|_{L^{2}}^{\frac{1}{2}}\left\|\nabla \mathbf{u}_{t}\right\|_{L^{2}}\|\nabla \theta\|_{L^{2}} \\
& \leq 2 \mu \frac{d}{d t} \int|\mathfrak{D}(\mathbf{u})|^{2} \theta d x+C\left\|\nabla \mathbf{u}_{t}\right\|_{L^{2}}^{2}+C\|\nabla \theta\|_{L^{2}}^{2},
\end{aligned}
\end{equation}
and
\begin{equation} \label{j3}
\begin{aligned}
J_{3} &=\nu \frac{d}{d t} \int|\nabla \times \mathbf{H}|^{2} \theta d x- \nu \int\left(|\nabla \times \mathbf{H}|^{2}\right)_{t} \theta d x \\
& \leq \nu \frac{d}{d t} \int|\nabla \times \mathbf{H}|^{2} \theta d x+C \int\left|\nabla \mathbf{H} \| \nabla \mathbf{H}_{t}\right| \theta d x \\
& \leq \nu \frac{d}{d t} \int|\nabla \times \mathbf{H}|^{2} \theta d x+C\|\nabla \mathbf{H}\|_{L^{3}}\left\|\nabla \mathbf{H}_{t}\right\|_{L^{2}}\|\theta\|_{L^{6}} \\
& \leq \nu \frac{d}{d t} \int|\nabla \times \mathbf{H}|^{2} \theta d x+C\|\nabla \mathbf{H}\|_{L^{2}}^{\frac{1}{2}}\left\|\nabla^{2} \mathbf{H}\right\|_{L^{2}}^{\frac{1}{2}}\left\|\nabla \mathbf{H}_{t}\right\|_{L^{2}}\|\nabla \theta\|_{L^{2}} \\
& \leq  \nu \frac{d}{d t} \int|\nabla \times \mathbf{H}|^{2} \theta d x+C\left\|\nabla \mathbf{H}_{t}\right\|_{L^{2}}^{2}+C\|\nabla \theta\|_{L^{2}}^{2}.
\end{aligned}
\end{equation}

Substituting \eqref{j1}-\eqref{j3} into \eqref{thetat}, one has
\begin{equation}
\begin{aligned}
& \frac{d}{dt} \int \left(\kappa |\nabla \theta|^2 -4\mu |\mathfrak{D}(\mathbf{u})|^2 \theta -2 \nu |\nabla \times \mathbf{H}|^2 \theta \right) dx + c_v \int \rho \theta_t^2 dx  \\
 \le &  C(\|\nabla \mathbf{u}_t\|_{L^2}^2 + \|\nabla \mathbf{H}_t\|_{L^2}^2) + C \|\nabla \theta\|_{L^2}^2,
\end{aligned}
\end{equation}
which together with the fact
\begin{equation}
\begin{aligned}
& \int \left( 4\mu |\mathfrak{D}(\mathbf{u})|^{2} + 2\nu |\nabla \times \mathbf{H}|^2\right) \theta d x \\
\leq & C\|\theta\|_{L^{6}}( \|\nabla \mathbf{u}\|_{L^{\frac{12}{5}}}^{2} +  \|\nabla \mathbf{H}\|_{L^{\frac{12}{5}}}^{2})\\
\leq & C\|\nabla \theta\|_{L^{2}}( \|\nabla \mathbf{u}\|_{L^{2}}^{\frac{3}{2}}\left\|\nabla^{2} \mathbf{u}\right\|_{L^{2}}^{\frac{1}{2}}  + \|\nabla \mathbf{H}\|_{L^{2}}^{\frac{3}{2}}\left\|\nabla^{2} \mathbf{H}\right\|_{L^{2}}^{\frac{1}{2}})\\
 \leq & \frac{\kappa}{2}\|\nabla \theta\|_{L^{2}}^{2}+C,
\end{aligned}
\end{equation}
and Gronwall's inequality and \eqref{uHtteq} yields 
\begin{equation} \label{thetateq1}
\sup_{0 \le t \le T} \|\nabla \theta\|_{L^2}^2 + \int_0^T \|\sqrt{\rho} \theta_t\|_{L^2}^2 dt \le C.
\end{equation}
Hence, the desired \eqref{rhothetaeq} follows from $\eqref{uHtteq}$ and \eqref{thetateq1}. Therefore the proof of lemma \ref{rhotheta} is completed.
\end{proof}

\begin{lemma} \label{lem7}
Under the assumption \eqref{blowup}, it holds that for any $0 \le T < T^*,$
\begin{equation} \label{nablathetat}
\sup_{0 \le t \le T} \|\sqrt{\rho} \theta_t \|_{L^2}^2 + \int_0^T \|\nabla \theta_t\|_{L^2}^2 dt \le C.
\end{equation}
\end{lemma}

\begin{proof}
Differentiating $\eqref{MHD_1}_3$ with respect to $t$ and direct computing gives
\begin{equation} \label{nablatheta1}
\begin{aligned}
c_v(\rho \theta_{tt} + \rho \mathbf{u}\cdot \nabla \theta_t) -\kappa \Delta \theta_t   = & -c_v \rho_t (\theta_t + \mathbf{u}\cdot \nabla \theta_t)- c_v \rho (\mathbf{u}_t\cdot \nabla) \theta  \\
& + 2\mu \left(|\mathfrak{D}(\mathbf{u})|^2\right)_t + \nu \left(|\nabla \times \mathbf{H}|^2\right)_t. 
\end{aligned}
\end{equation}
Multiplying \eqref{nablatheta1} by $\theta_t$ and integrating the resulting equation over $\mathbb{R}^3$ yield
\begin{equation} \label{nablathetaeq}
\begin{aligned}
\frac{c_{v}}{2} & \frac{d}{d t} \int \rho\left|\theta_{t}\right|^{2} d x+\kappa \int\left|\nabla \theta_{t}\right|^{2} d x \\
\quad=& c_{v} \int \operatorname{div}(\rho \mathbf{u})\left|\theta_{t}\right|^{2} d x+c_{v} \int \operatorname{div}(\rho \mathbf{u}) (\mathbf{u} \cdot \nabla \theta) \theta_{t} d x \\
&-c_{v} \int \rho \mathbf{u}_{t} \cdot \nabla \theta \theta_{t} d x+2 \mu \int\left(|\mathfrak{D}(\mathbf{u})|^{2}\right)_{t} \theta_{t} d x \\
& +  \nu \int\left(|\nabla \times \mathbf{H}|^{2}\right)_{t} \theta_{t} d x \triangleq \sum_{i=1}^{5} K_{i} .
\end{aligned}
\end{equation}
Next, we deal carefully with each term $K_1$-$K_5$ as follows: 
\begin{equation} \label{k1}
\begin{aligned}
K_{1} & \leq C \int \rho\left| \mathbf{u}\left\|\theta_{t}\right\| \nabla \theta_{t}\right| d x \\
& \leq C\|\mathbf{u}\|_{L^{\infty}}\left\|\sqrt{\rho} \theta_{t}\right\|_{L^{2}}\left\|\nabla \theta_{t}\right\|_{L^{2}} \\
& \leq C\left(\|\mathbf{u}\|_{L^{6}}+\|\nabla \mathbf{u}\|_{L^{6}}\right)\left\|\sqrt{\rho} \theta_{t}\right\|_{L^{2}}\left\|\nabla \theta_{t}\right\|_{L^{2}} \\
& \leq C\left(\|\nabla \mathbf{u}\|_{L^{2}}+\left\|\nabla^{2} \mathbf{u}\right\|_{L^{2}}\right)\left\|\sqrt{\rho} \theta_{t}\right\|_{L^{2}}\left\|\nabla \theta_{t}\right\|_{L^{2}} \\
& \leq C\left\|\sqrt{\rho} \theta_{t}\right\|_{L^{2}}^{2}+\frac{\kappa}{10}\left\|\nabla \theta_{t}\right\|_{L^{2}}^{2},
\end{aligned}
\end{equation}
\begin{equation} \label{k2}
\begin{aligned}
K_{2} \leq & C \int \left(\rho|\mathbf{u}||\nabla \mathbf{u}| |\nabla \theta||\theta_{t}| +\rho|\mathbf{u}|^{2}|\nabla^{2} \theta|| \theta_{t}| +\rho|\mathbf{u}|^{2}|\nabla \theta| | \nabla \theta_{t}|\right) d x \\
\leq & C\|\mathbf{u}\|_{L^{\infty}}\|\nabla \mathbf{u}\|_{L^{3}}\|\nabla \theta\|_{L^{2}}\left\|\theta_{t}\right\|_{L^{6}}+C\|\mathbf{u}\|_{L^{6}}^{2}\left\|\nabla^{2} \theta\right\|_{L^{2}}\left\|\theta_{t}\right\|_{L^{6}} \\
&+C\|\mathbf{u}\|_{L^{\infty}}^{2}\|\nabla \theta\|_{L^{2}}\left\|\nabla \theta_{t}\right\|_{L^{2}} \\
\leq & C\left\|\nabla \theta_{t}\right\|_{L^{2}}+C\left\|\nabla^{2} \theta\right\|_{L^{2}}\left\|\nabla \theta_{t}\right\|_{L^{2}} \\
\leq & C\left\|\nabla^{2} \theta\right\|_{L^{2}}^{2}+\frac{\kappa}{10}\left\|\nabla \theta_{t}\right\|_{L^{2}}^{2}+C,
\end{aligned}
\end{equation}
\begin{equation} \label{k3}
\begin{aligned}
K_{3} & \leq C\left\|\sqrt{\rho} \mathbf{u}_{t}\right\|_{L^{2}}\|\nabla \theta\|_{L^{3}}\left\|\theta_{t}\right\|_{L^{6}} \\
& \leq C\|\nabla \theta\|_{L^{2}}^{\frac{1}{2}}\left\|\nabla^{2} \theta\right\|_{L^{2}}^{\frac{1}{2}}\left\|\nabla \theta_{t}\right\|_{L^{2}} \\
& \leq C\left\|\nabla^{2} \theta\right\|_{L^{2}}^{2}+\frac{\kappa}{10}\left\|\nabla \theta_{t}\right\|_{L^{2}}^{2}+C,
\end{aligned}
\end{equation}
\begin{equation} \label{k4}
\begin{aligned}
K_{4} & \leq C \int\left|\nabla \mathbf{u}\left\|\nabla \mathbf{u}_{t}\right\| \theta_{t}\right| d x \\
& \leq C\|\nabla \mathbf{u}\|_{L^{3}}\left\|\nabla \mathbf{u}_{t}\right\|_{L^{2}}\left\|\theta_{t}\right\|_{L^{6}} \\
& \leq C\|\nabla \mathbf{u}\|_{L^{3}}\left\|\nabla \mathbf{u}_{t}\right\|_{L^{2}}\left\|\nabla \theta_{t}\right\|_{L^{2}} \\
& \leq C\left\|\nabla \mathbf{u}_{t}\right\|_{L^{2}}^{2}+\frac{\kappa}{10}\left\|\nabla \theta_{t}\right\|_{L^{2}}^{2},
\end{aligned}
\end{equation}
and 
\begin{equation} \label{k5}
\begin{aligned}
K_{5} & \leq C \int\left|\nabla \mathbf{H}\left\|\nabla \mathbf{H}_{t}\right\| \theta_{t}\right| d x \\
& \leq C\|\nabla \mathbf{H}\|_{L^{3}}\left\|\nabla \mathbf{H}_{t}\right\|_{L^{2}}\left\|\theta_{t}\right\|_{L^{6}} \\
& \leq C\|\nabla \mathbf{H}\|_{L^{3}}\left\|\nabla \mathbf{H}_{t}\right\|_{L^{2}}\left\|\nabla \theta_{t}\right\|_{L^{2}} \\
& \leq C\left\|\nabla \mathbf{H}_{t}\right\|_{L^{2}}^{2}+\frac{\kappa}{10}\left\|\nabla \theta_{t}\right\|_{L^{2}}^{2} .
\end{aligned}
\end{equation}
Furthermore, in view of $L^2$-estimate to Eq. $\eqref{MHD_1}_3$, we have
\begin{equation} \label{nabla2theta}
\begin{aligned}
\left\|\nabla^{2} \theta\right\|_{L^{2}}^{2} & \leq C\left(\left\|\rho \theta_{t}\right\|_{L^{2}}^{2}+\|\rho \mathbf{u} \cdot \nabla \theta\|_{L^{2}}^{2}+\left\||\nabla \mathbf{u}|^{2}\right\|_{L^{2}}^{2}+\left\||\nabla \mathbf{H}|^{2}\right\|_{L^{2}}^{2}\right) \\
& \leq C\left\|\sqrt{\rho} \theta_{t}\right\|_{L^{2}}^{2}+\|\mathbf{u}\|_{L^{\infty}}^{2}\|\nabla \theta\|_{L^{2}}^{2}+C\|\nabla \mathbf{u}\|_{L^{4}}^{4} + C\|\nabla \mathbf{H}\|_{L^{4}}^{4}\\
& \leq C\left\|\sqrt{\rho} \theta_{t}\right\|_{L^{2}}^{2}+C\|\nabla \mathbf{u}\|_{L^{2}}\left\|\nabla^{2} \mathbf{u}\right\|_{L^{2}}^{3}+C\|\nabla \mathbf{H}\|_{L^{2}}\left\|\nabla^{2} \mathbf{H} \right\|_{L^{2}}^{3} + C \\
& \leq C\left\|\sqrt{\rho} \theta_{t}\right\|_{L^{2}}^{2}+C,
\end{aligned}
\end{equation}
due to \eqref{GN1}, \eqref{rhothetaeq}.

Substituting \eqref{k1}-\eqref{k5} into \eqref{nablathetaeq}, we obtain
\begin{equation} \label{nablathetaeqq}
c_{v} \frac{d}{d t} \int \rho\left|\theta_{t}\right|^{2} d x+\kappa \int\left|\nabla \theta_{t}\right|^{2} d x \leq C\left\|\nabla \mathbf{u}_{t}\right\|_{L^{2}}^{2}+C\left\|\nabla \mathbf{H}_{t}\right\|_{L^{2}}^{2} + C\left\|\sqrt{\rho} \theta_{t}\right\|_{L^{2}}^{2}+C,
\end{equation}
which together with Gronwall's inequality and \eqref{rhothetaeq} yields the desired \eqref{nablathetat}.

Therefore, the proof of lemma \ref{lem7} is completed.
\end{proof}

\begin{lemma} \label{lem8}
For $\tilde{q} \in (3, 6]$, under the assumption \eqref{blowup}, it holds that for $0 \le T < T^*,$
\begin{equation} \label{365}
\begin{aligned}
 \sup_{0 \le t \le T} (\|\rho\|_{H^1 \cap W^{1, \tilde{q}}} + \|\nabla^2 \mathbf{u}\|_{L^2}^2 & + \|\nabla^2 \mathbf{H}\|_{L^2}^2 + \|\nabla^2 \theta \|_{L^2}^2)  \\
 & +  \int_0^T (\|\nabla^2 \mathbf{u}\|_{L^{\tilde{q}}}^2  + \|\nabla^2 \mathbf{H}\|_{L^{\widetilde{q}}}^2 + \|\nabla^2 \theta\|_{L^{\tilde{q}}}^2) dt \le C.
\end{aligned}
\end{equation}
\end{lemma}

\begin{proof}
Firstly, it follows from \eqref{H22}, \eqref{rhothetaeq}, \eqref{nablathetat} and \eqref{nabla2theta} that
\begin{equation}  \label{366}
\|\nabla^2 \mathbf{u}\|_{L^2}^2 + \|\nabla^2 \mathbf{H}\|_{L^2}^2 + \|\nabla^2 \theta\|_{L^2}^2 \le C.
\end{equation}
Next, the first order spatial derivatives $\partial_i \rho, i=1, 2, 3$ satisfy
\begin{equation*}
\partial_t (\partial_i \rho) + \mathbf{u}\cdot \nabla (\partial_i \rho) + (\partial_i \mathbf{u})\cdot \nabla \rho =0.
\end{equation*}
Therefore, for any $q \in (3, 6],$ standard energy method gives 
\begin{equation} \label{ff}
\begin{aligned}
\frac{d}{dt} \|\nabla \rho\|_{L^q} & \le C \|\nabla \mathbf{u}\|_{L^{\infty}} \|\nabla \rho\|_{L^q} \\
& \le C \|\nabla \mathbf{u}\|_{W^{1, 6}} \|\nabla \rho\|_{L^q} \\
& \le C (\|\nabla \mathbf{u}\|_{L^6} + \|\nabla^2 \mathbf{u}\|_{L^6}) \|\nabla \rho\|_{L^q} \\
& \le C (1 + \|\rho \mathbf{u}_t\|_{L^6} + \|\rho \mathbf{u}\cdot \nabla \mathbf{u}\|_{L^6}+\|\mathbf{H}\cdot \nabla \mathbf{H}\|_{L^6})\|\nabla \rho\|_{L^q} \\
& \le C (1 + \|\nabla \mathbf{u}_t\|_{L^2} + \|\mathbf{u}\|_{L^{\infty}}\|\nabla \mathbf{u}\|_{L^6} + \|\mathbf{H}\|_{L^{\infty}}\|\nabla \mathbf{H}\|_{L^6})\|\nabla \rho\|_{L^q} \\
& \le C (1 + \|\nabla \mathbf{u}_t\|_{L^2} + (\|\mathbf{u}\|_{6}+ \|\nabla \mathbf{u}\|_{L^6})\|\nabla^2 \mathbf{u}\|_{L^2} \\
& \quad +(\|\mathbf{H}\|_{6}+ \|\nabla \mathbf{H}\|_{L^6})\|\nabla^2 \mathbf{H}\|_{L^2})\|\nabla \rho\|_{L^q} \\
& \le C (1 + \|\nabla \mathbf{u}_t\|_{L^2} + \|\nabla^2 \mathbf{u}\|_{L^2} +\|\nabla^2 \mathbf{H}\|_{L^2})\|\nabla \rho\|_{L^q} \\
& \le C (1 + \|\nabla \mathbf{u}_t\|_{L^2}^2)\|\nabla \rho\|_{L^q},
\end{aligned}
\end{equation}
due to \eqref{GN1}, \eqref{GN2}, \eqref{stokesin}, \eqref{nablauH_1}, \eqref{rhothetaeq}, \eqref{366} and interpolation inequality.

Together with Gronwall's inequality and taking $q= 2, \tilde{q}$ in \eqref{ff}, we have
\begin{equation} \label{368}
\sup_{0 \le t \le T} \|\nabla \rho\|_{L^2 \cap L^{\tilde{q}}} \le C.
\end{equation}
Next, in view of \eqref{GN1}, \eqref{GN2}, \eqref{stokesin}, \eqref{nablauH_1}, \eqref{rhothetaeq} and \eqref{366}, we obtain{}
\begin{equation} \label{369}
\begin{aligned}
\int_0^T \left\|\nabla^{2} \mathbf{u}\right\|_{L^{\tilde{q}}}^{2} d t  \leq  & C \int_0^T \left(\left\|\rho \mathbf{u}_{t}\right\|_{L^{\tilde{q}}}^{2}+\|\rho \mathbf{u} \cdot \nabla \mathbf{u}\|_{L^{\tilde{q}}}^{2} + \|\mathbf{H}\cdot \nabla \mathbf{H}\|_{L^{\tilde{q}}}^{2} \right) d t \\
\leq & C \int_0^T \left(\left\|\rho \mathbf{u}_{t}\right\|_{L^{2}}^{2}+\left\|\rho \mathbf{u}_{t}\right\|_{L^{6}}^{2}+\|\mathbf{u}\|_{L^{\infty}}^{2}\|\nabla \mathbf{u}\|_{L^{\tilde{q}}}^{2}+ \|\mathbf{H}\|_{L^{\infty}}^{2}\|\nabla \mathbf{H}\|_{L^{\tilde{q}}}^{2}\right) d t \\
\leq & C \int_0^T \left(1+\left\|\nabla \mathbf{u}_{t}\right\|_{L^{2}}^{2}+\left(\|\mathbf{u}\|_{L^{6}}^{2}+\|\nabla \mathbf{u}\|_{L^{6}}^{2}\right)\left(\|\nabla \mathbf{u}\|_{L^{2}}^{2}+\left\|\nabla^{2} \mathbf{u}\right\|_{L^{2}}^{2}\right)\right.\\
&\left.\times \left(\|\mathbf{H}\|_{L^{6}}^{2}+\|\nabla \mathbf{H}\|_{L^{6}}^{2}\right)\left(\|\nabla \mathbf{H}\|_{L^{2}}^{2}+\left\|\nabla^{2} \mathbf{H}\right\|_{L^{2}}^{2}\right)\right) d t \\
\leq & C \int_0^T \left(1+\left\|\nabla \mathbf{u}_{t}\right\|_{L^{2}}^{2}+\left(\|\nabla \mathbf{u}\|_{L^{2}}^{2}+\|\nabla^2 \mathbf{u}\|_{L^{2}}^{2}\right)\left(\|\nabla \mathbf{u}\|_{L^{2}}^{2}+\left\|\nabla^{2} \mathbf{u}\right\|_{L^{2}}^{2}\right)\right.\\
&\left.\times \left(\|\nabla \mathbf{H}\|_{L^{2}}^{2}+\|\nabla^2 \mathbf{H}\|_{L^{2}}^{2}\right)\left(\|\nabla \mathbf{H}\|_{L^{2}}^{2}+\left\|\nabla^{2} \mathbf{H}\right\|_{L^{2}}^{2}\right)\right) d t \\
\leq & C.
\end{aligned}
\end{equation}
Similarly, 
\begin{equation} \label{370}
\begin{aligned}
\int_0^T \left\|\nabla^{2} \mathbf{H}\right\|_{L^{\tilde{q}}}^{2} d t  \leq  & C \int_0^T \left(\left\|\mathbf{H}_{t}\right\|_{L^{\tilde{q}}}^{2}+\|\mathbf{u} \cdot \nabla \mathbf{H}\|_{L^{\tilde{q}}}^{2} + \|\mathbf{H} \cdot \nabla \mathbf{u}\|_{L^{\tilde{q}}}^{2} \right) d t \\
\leq & C \int_0^T \left(\left\|\mathbf{H}_{t}\right\|_{L^{2}}^{2}+\left\|\mathbf{H}_{t}\right\|_{L^{6}}^{2}+\|\mathbf{u}\|_{L^{\infty}}^{2}\|\nabla \mathbf{H}\|_{L^{\tilde{q}}}^{2}+ \|\mathbf{H}\|_{L^{\infty}}^{2}\|\nabla \mathbf{u}\|_{L^{\tilde{q}}}^{2}\right) d t \\
\leq & C \int_0^T \left(1+\left\|\nabla \mathbf{H}_{t}\right\|_{L^{2}}^{2}+\left(\|\mathbf{u}\|_{L^{6}}^{2}+\|\nabla \mathbf{u}\|_{L^{6}}^{2}\right)\left(\|\nabla \mathbf{H}\|_{L^{2}}^{2}+\left\|\nabla^{2} \mathbf{H}\right\|_{L^{2}}^{2}\right)\right.\\
&\left.\times \left(\|\mathbf{H}\|_{L^{6}}^{2}+\|\nabla \mathbf{H}\|_{L^{6}}^{2}\right)\left(\|\nabla \mathbf{u}\|_{L^{2}}^{2}+\left\|\nabla^{2} \mathbf{u}\right\|_{L^{2}}^{2}\right)\right) d t \\
\leq & C \int_0^T \left(1+\left\|\nabla \mathbf{H}_{t}\right\|_{L^{2}}^{2}+\left(\|\nabla \mathbf{u}\|_{L^{2}}^{2}+\|\nabla^2 \mathbf{u}\|_{L^{2}}^{2}\right)\left(\|\nabla \mathbf{H}\|_{L^{2}}^{2}+\left\|\nabla^{2} \mathbf{H}\right\|_{L^{2}}^{2}\right)\right.\\
&\left.\times \left(\|\nabla \mathbf{H}\|_{L^{2}}^{2}+\|\nabla^2 \mathbf{H}\|_{L^{2}}^{2}\right)\left(\|\nabla \mathbf{u}\|_{L^{2}}^{2}+\left\|\nabla^{2} \mathbf{u}\right\|_{L^{2}}^{2}\right)\right) d t \\
\leq & C.
\end{aligned}
\end{equation}
Furthermore, using the standard $L^p$-estimate to elliptic equation $\eqref{MHD_1}_3$, we obtain
\begin{equation} \label{371}
\begin{aligned}
& \int_0^T \left\| \nabla^{2} \theta \right\|_{L^{\tilde{q}}}^2 d t  \\
\leq & C \int_0^T \left(\left\|\rho \theta_{t}\right\|_{L^{\tilde{q}}}^{2}+\|\rho \mathbf{u} \cdot \nabla \theta\|_{L^{\tilde{q}}}^{2}+\left\||\nabla \mathbf{u}|^{2}\right\|_{L^{\tilde{q}}}^{2}+ \left\||\nabla \mathbf{H}|^{2}\right\|_{L^{\tilde{q}}}^{2}\right) d t \\
\leq & C \int_0^T (\left\|\rho \theta_{t}\right\|_{L^{2}}^{2}+\left\|\rho \theta_{t}\right\|_{L^{6}}^{2}+\|\mathbf{u}\|_{L^{\infty}}^{2}\|\nabla \theta\|_{L^{q}}^{2} \\
&\quad \quad +\|\nabla \mathbf{u}\|_{L^{\infty}}^{2}\|\nabla \mathbf{u}\|_{L^{\tilde{q}}}^{2} + \|\nabla \mathbf{H}\|_{L^{\infty}}^{2}\|\nabla \mathbf{H}\|_{L^{\tilde{q}}}^{2} ) dt \\
\leq & C \int_0^T (1+\left\|\nabla \theta_{t}\right\|_{L^{2}}^{2}+ \left\|\nabla^{2} \theta\right\|_{L^{2}}^{2}+\left(\|\nabla \mathbf{u}\|_{L^{6}}^{2}+\left\|\nabla^{2} \mathbf{u}\right\|_{L^{6}}^{2} \right)\left(\|\nabla \mathbf{u}\|_{L^{2}}^{2}+\|\nabla \mathbf{u}\|_{L^{6}}^{2}\right) \\
& \quad \quad + \left(\|\nabla \mathbf{H}\|_{L^{6}}^{2}+\left\|\nabla^{2} \mathbf{H}\right\|_{L^{6}}^{2} \right)\left(\|\nabla \mathbf{H}\|_{L^{2}}^{2}+\|\nabla \mathbf{H}\|_{L^{6}}^{2}\right))d t \\
\leq & C \int\left(1+\left\|\nabla \theta_{t}\right\|_{L^{2}}^{2}+\left\|\nabla \mathbf{u}_{t}\right\|_{L^{2}}^{2} + \left\|\nabla \mathbf{H}_t\right\|_{L^2}^2 \right) d t \\
\leq & C .
\end{aligned}
\end{equation}
Thus, in view of \eqref{366}, \eqref{368}-\eqref{371}, we complete the proof of Lemma \ref{lem8}.

\end{proof}

\begin{remark}
Define $\dot{f}$ the material derivative of function $f$ with $\dot{f} \triangleq f_t + \mathbf{u} \cdot \nabla f.$ Then we can derive the regularity of the terms $\rho \dot{\mathbf{u}}$ and $ \rho \dot{\theta}$ for the later analysis. 
Indeed, one can deduce from \eqref{stokesin}, \eqref{nablauH_1}, \eqref{rhothetaeq} and \eqref{365} that
\begin{equation}
\begin{aligned}
\|\rho \dot{\mathbf{u}}\|_{L^2}^2 & = \|\rho (\mathbf{u}_t + \mathbf{u} \cdot \nabla \mathbf{u}) \|_{L^2}^2 \\
& \le C (\|\rho \mathbf{u}_t\|_{L^2}^2 + \|\rho \mathbf{u}\cdot \nabla \mathbf{u}\|_{L^2}^2) \\
& \le C (\|\rho \mathbf{u}_t\|_{L^2}^2 + \|\nabla \mathbf{u}\|_{L^2}^3 \|\nabla^2 \mathbf{u}\|_{L^2}) \\
& \le C,
\end{aligned}
\end{equation}
and 
\begin{equation}
\begin{aligned}
\|\rho \dot{\theta}\|_{L^2}^2 & = \|\rho (\theta_t + \mathbf{u} \cdot \nabla \theta) \|_{L^2}^2 \\
& \le C (\|\rho \theta_t\|_{L^2}^2 + \|\mathbf{u}\|_{L^{\infty}}^2 \|\nabla \theta\|_{L^2}^2) \\
& \le C (\|\rho \theta_t\|_{L^2}^2 + \|\nabla \mathbf{u}\|_{H^1}^2 \|\nabla \theta\|_{L^2}^2) \\
& \le C.
\end{aligned}
\end{equation}
\end{remark}

\noindent {\bf{Proof of Theorem \ref{thm1}}.} With the aid of the some a priori estimates Lemmas \ref{lem1}-\ref{lem8}, we can now prove Theorem \ref{thm1} as follows.

On the one hand, the functions $(\rho, \mathbf{u}, \mathbf{H}, \theta)(x, T^*) = \lim\limits_{t \to T^*} (\rho, \mathbf{u}, \mathbf{H}, \theta)(x, t)$ satisfy the regularity condition on the initial data at time $t=T^*.$ Furthermore, standard arguments yield $( \rho \dot{\mathbf{u}}, \rho \dot{\theta}) \in C([0, T^*]; L^2),$ which implies
$$(\rho \dot{\mathbf{u}}, \rho \dot{\theta})(x, T^*) = \lim_{t \to T^*} (\rho \dot{\mathbf{u}}, \rho \dot{\theta}) (x, t) \in L^2. $$
Hence,
$$(-\mathrm{div}(2\mu\mathfrak{D}(\mathbf{u}))+\nabla P-\mathbf{H}\cdot \nabla \mathbf{H})|_{t=T^*} = \sqrt{\rho}(x, T^*)\tilde{\mathbf{g}}_1(x),$$
$$(\kappa \Delta \theta + 2\mu |\mathfrak{D}(\mathbf{u})|^2 + \nu|\nabla \times \mathbf{H}|^2)|_{t=T^*} = \sqrt{\rho}(x, T^*)\tilde{\mathbf{g}}_2(x),$$
where
\begin{equation*}
\tilde{\mathbf{g}}_1(x) \triangleq \left\{
\begin{aligned}
\rho^{-\frac{1}{2}}(x, T^*)(\rho \dot{\mathbf{u}})(x, T^*), \quad & \text{for}~x \in \{x|\rho(x, T^*)>0\}, \\
0, \quad & \text{for}~x \in \{x|\rho(x, T^*)=0\}, \\
\end{aligned}
\right.
\end{equation*}
and 
\begin{equation*}
\tilde{\mathbf{g}}_2(x) \triangleq \left\{
\begin{aligned}
c_v \rho^{-\frac{1}{2}}(x, T^*)(\rho \dot{\theta})(x, T^*), \quad & \text{for}~x \in \{x|\rho(x, T^*)>0\}, \\
0, \quad & \text{for}~x \in \{x|\rho(x, T^*)=0\}, \\
\end{aligned}
\right.
\end{equation*}
satisfying $\tilde{\mathbf{g}}_1, \tilde{\mathbf{g}}_2 \in L^2$. Thus, $(\rho, \mathbf{u}, \mathbf{H}, \theta)(x, T^*)$ satisfies compatibility condition. Therefore, we can take $(\rho, \mathbf{u}, \mathbf{H}, \theta)(x, T^*)$ as the initial data and apply Lemma \ref{local} again to extend the local strong solutions beyond $T^*$, which contradicts the assumption that $T^*$ is the maximal existence time of strong solutions. Therefore, we complete the proof of Theorem \ref{thm1}.

\section{Proof of Theorem \ref{thm2}}

Throughout this section, we denote 
$$ C_0 \triangleq \|\sqrt{\rho_0} \mathbf{u}_0\|_{L^2}^2 + \|\mathbf{H}_0\|_{L^2}^2.$$

Firstly, applying \cite[Theorem 2.1]{lions} and integrating \eqref{energy_diff} with respect to $t$ respectively, we have the following results.

\begin{lemma} \label{lem41}
Let $(\rho, \mathbf{u}, \mathbf{H}, \theta)$ be a strong solution to the system \eqref{MHD}-\eqref{boundary} on $(0, T)$. Then for any $t \in (0, T)$, it holds that
\begin{equation} \label{rhobdd}
\|\rho (t) \|_{L^{\infty}} = \|\rho_0\|_{L^{\infty}},
\end{equation}
and
\begin{equation} \label{energy4}
\begin{aligned}
\|\sqrt{\rho} \mathbf{u}(t)\|_{L^2}^2 + \|\mathbf{H}(t)\|_{L^2}^2 + 2\int_0^t \left(\mu \|\nabla \mathbf{u}\|_{L^2}^2 + \nu \|\nabla \mathbf{H}\|_{L^2}^2 \right) ds \le C_0.
\end{aligned}
\end{equation}
\end{lemma}

\begin{lemma} \label{lem42}
Let $(\rho, \mathbf{u}, \mathbf{H}, \theta)$ be a strong solution to the system \eqref{MHD}-\eqref{boundary} on $(0, T).$ Then for any $t \in (0, T)$, it holds that
\begin{equation} \label{nablauH4}
\begin{aligned}
 \sup_{0 \le s \le t} (\mu\|\nabla \mathbf{u}\|_{L^2}^2 + \nu\|\nabla \mathbf{H}\|_{L^2}^2) & \le 2 \left( \mu\|\nabla \mathbf{u}_0\|_{L^2}^2 + \nu\|\nabla \mathbf{H}_0\|_{L^2}^2 \right) \\
& + C\sqrt{C_0} \sup_{0 \le s\le t} \|\nabla \mathbf{H}\|_{L^2}^3 + 
C C_0 \sup_{0 \le s\le t} (\|\nabla \mathbf{u}\|_{L^2}^4+ \|\nabla \mathbf{H}\|_{L^2}^4),
\end{aligned}
\end{equation}
where (and in what follows) $C$ denotes a generic positive constant depending only on $\mu, \nu$ and $\|\rho_0\|_{L^{\infty}}$.
\end{lemma}

\begin{proof}
Multiplying $\eqref{MHD_1}_2$ by $\mathbf{u}_t$, $\eqref{MHD_1}_4$ multiplied by $\mathbf{H}_t$, and integrating the resulting equality over $\mathbb{R}^3$, we obtain from Cauchy-Schwarz inequality that
\begin{equation} \label{eq41}
\begin{aligned}
&\frac{1}{2}\frac{d}{dt} \int (\mu |\nabla \mathbf{u}|^2 + \nu |\nabla \mathbf{H}|^2 ) dx  + \int (\rho |\mathbf{u}_t|^2 + |\mathbf{H}_t|^2 ) dx \\
= & \int \mathbf{H}\cdot \nabla \mathbf{H} \cdot \mathbf{u}_t dx - \int \rho \mathbf{u}\cdot \nabla \mathbf{u} \cdot \mathbf{u}_t dx + \int (\mathbf{H}\cdot \nabla \mathbf{u}- \mathbf{u}\cdot \nabla \mathbf{H}) \cdot H_t dx \\
= & -\frac{d}{dt} \int \mathbf{H} \cdot \nabla \mathbf{u} \cdot \mathbf{H} dx + \int \mathbf{H}_t \cdot \nabla \mathbf{u} \cdot \mathbf{H} dx  + \int \mathbf{H} \cdot \nabla \mathbf{u} \cdot \mathbf{H}_t  dx \\
& - \int \rho \mathbf{u} \cdot \nabla \mathbf{u} \cdot \mathbf{u}_t dx + \int (\mathbf{H}\cdot \nabla \mathbf{u}- \mathbf{u}\cdot \nabla \mathbf{H}) \cdot \mathbf{H}_t dx \\
\le  & -\frac{d}{dt} \int \mathbf{H} \cdot \nabla \mathbf{u} \cdot \mathbf{H} dx + \frac{1}{2} \int \left(\rho |\mathbf{u}_t|^2 + |\mathbf{H}_t|^2 \right) dx \\
& + C  \int \left( \rho |\mathbf{u}|^2  |\nabla \mathbf{u}|^2 + |\mathbf{H}|^2 |\nabla \mathbf{u}|^2 + |\mathbf{u}|^2 |\nabla \mathbf{H}|^2 \right)dx,
\end{aligned}
\end{equation}
which implies that
\begin{equation} \label{eq42}
\begin{aligned}
&\frac{d}{dt} \int (\mu |\nabla \mathbf{u}|^2 + \nu |\nabla \mathbf{H}|^2 + 2 \mathbf{H}\cdot \nabla \mathbf{u}\cdot \mathbf{H}) dx  + \|\sqrt{\rho}\mathbf{u}_t\|_{L^2}^2  + \|\mathbf{H}_t\|_{L^2} ^2  \\
\le  &  C  \int \left( \rho |\mathbf{u}|^2  |\nabla \mathbf{u}|^2 + |\mathbf{H}|^2 |\nabla \mathbf{u}|^2 + |\mathbf{u}|^2 |\nabla \mathbf{H}|^2 \right)dx.
\end{aligned}
\end{equation}
Integrating \eqref{eq42} with respect to the time variable over $(0, t)$ gives rise to
\begin{equation} \label{eq43}
\begin{aligned}
&\sup_{0 \le s \le t} \left( \mu \|\nabla \mathbf{u}\|_{L^2}^2 + \nu \|\nabla \mathbf{H}\|_{L^2}^2 \right)  + \int_0^t \left( \|\sqrt{\rho}\mathbf{u}_s\|_{L^2}^2  + \|\mathbf{H}_s\|_{L^2} ^2 \right) ds  \\
\le & ( \mu \|\nabla \mathbf{u}_0\|_{L^2}^2 + \nu \|\nabla \mathbf{H}_0 \|_{L^2}^2)  + 4 \sup_{0 \le s \le t} \int |\mathbf{H}|^2 |\nabla \mathbf{u}| dx \\
& +  C  \int_0^t \int \left( \rho |\mathbf{u}|^2  |\nabla \mathbf{u}|^2 + |\mathbf{H}|^2 |\nabla \mathbf{u}|^2 + |\mathbf{u}|^2 |\nabla \mathbf{H}|^2 \right)dx ds.
\end{aligned}
\end{equation}
Recall that $(\mathbf{u}, P)$ satisfies the following Stokes system:
\begin{equation}
\begin{cases}
-\mu \Delta \mathbf{u}+\nabla P=-\rho \mathbf{u}_t -\rho \mathbf{u}\cdot \nabla \mathbf{u} + \mathbf{H}\cdot \nabla \mathbf{H}, & x \in \mathbb{R}^3, \\ \operatorname{div} \mathbf{u} =0, & x \in \mathbb{R}^3, \\ \mathbf{u}(x) \rightarrow 0,  & |x| \rightarrow \infty. 
\end{cases}
\end{equation}
We thus obtain from \eqref{stokesin} that
\begin{equation}  \label{eq44}
\begin{aligned}
\|\nabla^2 \mathbf{u} \|_{L^2}^2 & \le C(\|\rho \mathbf{u}_t\|_{L^2}^2 + \|\rho \mathbf{u}\cdot \nabla \mathbf{u}\|_{L^2}^2 + \|\mathbf{H}\cdot \nabla \mathbf{H}\|_{L^2}^2) \\
& \le C (\|\sqrt{\rho} \mathbf{u}_t\|_{L^2}^2 + \|\sqrt{\rho} \mathbf{u}\cdot \nabla \mathbf{u}\|_{L^2}^2 + \|\mathbf{H}\cdot \nabla \mathbf{H}\|_{L^2}^2).
\end{aligned}
\end{equation}
Applying the classical $L^2$-estimates for elliptic system on $\mathbf{H}$ gives
\begin{equation} \label{eq45}
\begin{aligned}
\|\nabla^2 \mathbf{H}\|_{L^2}^2 & \le C(\|\mathbf{H}_t\|_{L^2}^2 + \|\mathbf{u}\cdot \nabla \mathbf{H}\|_{L^2}^2 + \|\mathbf{H}\cdot \nabla \mathbf{u}\|_{L^2}^2), \\
\end{aligned}
\end{equation}
which together with \eqref{eq44} leads to
\begin{equation} \label{eq46}
\begin{aligned}
 \|\nabla^2 \mathbf{u}\|_{L^2}^2 + \|\nabla^2 \mathbf{H}\|_{L^2}^2 & \le L(\|\sqrt{\rho} \mathbf{u}_t\|_{L^2}^2 + \|\mathbf{H}_t\|_{L^2}^2) \\
& + C( \|\sqrt{\rho} \mathbf{u}\cdot \nabla \mathbf{u}\|_{L^2}^2 + \|\mathbf{H}\cdot \nabla \mathbf{H}\|_{L^2}^2 + \|\mathbf{u}\cdot \nabla \mathbf{H}\|_{L^2}^2 + \|\mathbf{H}\cdot \nabla \mathbf{u}\|_{L^2}^2), \\
\end{aligned}
\end{equation}
for some positive constant $L$ depending only on $\mu, \nu$ and $\|\rho\|_{L^{\infty}}.$ Integrating \eqref{eq46} multiplied by $\frac{1}{2L}$ with respect to time variable over $(0, t)$ and adding the resulting inequality to \eqref{eq43}, we have
\begin{equation} \label{eq47}
\begin{aligned}
&\sup_{0 \le s \le t} \left( \mu \|\nabla \mathbf{u}\|_{L^2}^2 + \nu \|\nabla \mathbf{H}\|_{L^2}^2 \right)  + \frac{1}{2L} \int_0^t \left(\|\nabla^2 \mathbf{u}\|_{L^2}^2  + \|\nabla^2 \mathbf{H}\|_{L^2} ^2 \right) ds \\
& \quad + \frac{1}{2}\int_0^t \left( \|\sqrt{\rho}\mathbf{u}_s\|_{L^2}^2  + \|\mathbf{H}_s\|_{L^2} ^2 \right) ds  \\
& \le  (\mu \|\nabla \mathbf{u}_0\|_{L^2}^2 + \nu \|\nabla \mathbf{H}_0 \|_{L^2}^2)  + 4 \sup_{0 \le s \le t} \int |\mathbf{H}|^2 |\nabla \mathbf{u}| dx \\
& \quad +    \bar{L}  \int_0^t \int \left( \rho |\mathbf{u}|^2  |\nabla \mathbf{u}|^2 + |\mathbf{H}|^2 |\nabla \mathbf{u}|^2 + |\mathbf{u}|^2 |\nabla \mathbf{H}|^2 + |\mathbf{H}|^2 |\nabla \mathbf{H}|^2 \right)dx ds.
\end{aligned}
\end{equation}

By H\"older inequality, Sobolev inequality and \eqref{energy4}, we have
\begin{equation} \label{eq48}
\begin{aligned}
\int |\mathbf{H}|^2 |\nabla \mathbf{u}| dx & \le \|\mathbf{H}\|_{L^4}^2 \|\nabla \mathbf{u}\|_{L^2} \le \|\mathbf{H}\|_{L^2}^{\frac{1}{2}}\|\nabla \mathbf{H}\|_{L^2}^{\frac{3}{2}} \|\nabla \mathbf{u}\|_{L^2} \\
& \le \frac{\mu}{8} \|\nabla \mathbf{u}\|_{L^2}^2 + C \|\mathbf{H}\|_{L^2} \|\nabla \mathbf{H}\|_{L^2}^3 \\
& \le \frac{\mu}{8} \|\nabla \mathbf{u}\|_{L^2}^2 + C \sqrt{C_0} \|\nabla \mathbf{\mathbf{H}}\|_{L^2}^3,
\end{aligned}
\end{equation}
which gives rise to 
\begin{equation} \label{eq49}
4 \sup_{0 \le s \le t} \int |\mathbf{H}|^2 |\nabla \mathbf{u}| dx \le \frac{\mu}{2}\sup_{0 \le s \le t} \|\nabla \mathbf{u}\|_{L^2}^2 + C \sqrt{C_0} \sup_{0 \le s \le t} \|\nabla \mathbf{H}\|_{L^2}^3.
\end{equation}
In a similar way, we have
\begin{equation} \label{eq50}
\begin{aligned}
& \bar{L}  \int \left( \rho |\mathbf{u}|^2  |\nabla \mathbf{u}|^2 + |\mathbf{H}|^2 |\nabla \mathbf{u}|^2 + |\mathbf{u}|^2 |\nabla \mathbf{H}|^2 + |\mathbf{H}|^2 |\nabla \mathbf{H}|^2 \right)dx \\
 \le & \bar{L}\|\rho\|_{L^{\infty}}\|\mathbf{u}\|_{L^6}^2 \|\nabla \mathbf{u}\|_{L^2} \|\nabla \mathbf{u}\|_{L^6} + \bar{L} \|\mathbf{H}\|_{L^6}^2 \|\nabla \mathbf{u}\|_{L^2} \|\nabla u\|_{L^6} \\
& \quad + \bar{L}\|\mathbf{u}\|_{L^6}^2 \|\nabla \mathbf{H}\|_{L^2} \|\nabla \mathbf{H}\|_{L^6} + \bar{L}\|\mathbf{H}\|_{L^6}^2 \|\nabla \mathbf{H}\|_{L^2} \|\nabla H\|_{L^6} \\
 \le & C \|\nabla \mathbf{u}\|_{L^2}^3 \|\nabla^2 \mathbf{u}\|_{L^2} + C  \|\nabla \mathbf{H}\|_{L^2}^2 \|\nabla \mathbf{u}\|_{L^2} \|\nabla^2 \mathbf{u}\|_{L^2} \\
& \quad + C \|\nabla \mathbf{u}\|_{L^2}^2 \|\nabla \mathbf{H}\|_{L^2} \|\nabla^2 \mathbf{H}\|_{L^2} + C \|\nabla \mathbf{H}\|_{L^2}^3 \|\nabla \mathbf{H}\|_{L^6} \\
\le & \frac{1}{4L} (\|\nabla^2 \mathbf{u}\|_{L^2}^2 + \|\nabla^2 \mathbf{H}\|_{L^2}^2)  + C(\|\nabla \mathbf{u}\|_{L^2}^2 + \|\nabla \mathbf{H}\|_{L^2}^2)(\|\nabla \mathbf{u}\|_{L^2}^4 + \|\nabla \mathbf{H}\|_{L^2}^4). \\
\end{aligned}
\end{equation}
Integrating the above inequality \eqref{eq50} with respect to time variable over $(0, t)$ gives
\begin{equation} \label{eq51}
\begin{aligned}
& \bar{L} \int_0^t  \int \left( \rho |\mathbf{u}|^2  |\nabla \mathbf{u}|^2 + |\mathbf{H}|^2 |\nabla \mathbf{u}|^2 + |\mathbf{u}|^2 |\nabla \mathbf{H}|^2 + |\mathbf{H}|^2 |\nabla \mathbf{H}|^2 \right)dx ds\\ 
\le &  \frac{1}{4L} \int_0^t (\|\nabla^2 \mathbf{u}\|_{L^2}^2 + \|\nabla^2 \mathbf{H}\|_{L^2}^2) ds \\
& \quad +  C\sup_{0 \le s \le t} (\|\nabla \mathbf{u}\|_{L^2}^4 + \|\nabla \mathbf{H}\|_{L^2}^4) \int_0^t (\|\nabla \mathbf{u}\|_{L^2}^2 + \|\nabla \mathbf{H}\|_{L^2}^2)ds \\
\le &  \frac{1}{4L} \int_0^t (\|\nabla^2 \mathbf{u}\|_{L^2}^2 + \|\nabla^2 \mathbf{H}\|_{L^2}^2) ds \\
& \quad +  C\sup_{0 \le s \le t} (\|\nabla \mathbf{u}\|_{L^2}^4 + \|\nabla \mathbf{H}\|_{L^2}^4) \int_0^t (\mu \|\nabla \mathbf{u}\|_{L^2}^2 + \nu \|\nabla \mathbf{H}\|_{L^2}^2)ds \\
\le &  \frac{1}{4L} \int_0^t (\|\nabla^2 \mathbf{u}\|_{L^2}^2 + \|\nabla^2 \mathbf{H}\|_{L^2}^2) ds +  C  C_0 \sup_{0 \le s \le t} (\|\nabla \mathbf{u}\|_{L^2}^4 + \|\nabla \mathbf{H}\|_{L^2}^4),
\end{aligned}
\end{equation}
due to \eqref{energy4}.

Substituting \eqref{eq49} and \eqref{eq51} into \eqref{eq47} implies the desired \eqref{nablauH4} and therefore the proof of Lemma \ref{lem42} is completed.
\end{proof}

\begin{lemma} \label{lem3}
Let $(\rho, \mathbf{u}, \mathbf{H}, \theta)$ be a strong solution to the system \eqref{MHD}-\eqref{boundary} on $(0, T)$. Then there exists a positive constant $\varepsilon_0$ depending only on $\mu, \nu$ and $\|\rho_0\|_{L^\infty}$ such that for any $ t\in (0, T),$ it holds that
\begin{equation} \label{eq52}
\sup_{0 \le t \le T} (\mu\|\nabla \mathbf{u}\|_{L^2}^2 + \nu^2 \|\nabla \mathbf{H}\|_{L^2}^2) \le 8 (\mu\|\nabla \mathbf{u}_0\|_{L^2}^2 + \nu \|\nabla \mathbf{H}_0\|_{L^2}^2),
\end{equation}
provided that
\begin{equation} \label{eq53}
(\|\sqrt{\rho_0} \mathbf{u}_0\|_{L^2}^2 + \|\mathbf{H}_0\|_{L^2}^2)(\|\nabla \mathbf{u}_0\|_{L^2}^2 + \|\nabla \mathbf{H}_0\|_{L^2}^2)) \le \varepsilon_0.
\end{equation}
\end{lemma}

\begin{proof}
Define function $E(t)$ as follows
$$E(t) \triangleq \sup_{0 \le s \le t} (\mu\|\nabla \mathbf{u}\|_{L^2}^2 + \nu\|\nabla \mathbf{H}\|_{L^2}^2).$$
In view of the regularity of $\mathbf{u}$ and $\mathbf{H}$ as described in Lemma \ref{local}, it is easy to check that $E(t)$ is a continuous function on $[0, T].$ By \eqref{nablauH4}, there is a positive constant $M$ depending only on $\mu, \nu$ and $\|\rho_0\|_{L^{\infty}}$ such that
\begin{equation} \label{eq54}
\begin{aligned}
E(t) & \le 2 (\mu\|\nabla \mathbf{u}_0\|_{L^2}^2 + \nu \|\nabla \mathbf{H}_0\|_{L^2}^2)  + \sqrt{M}\sqrt{C_0}E^{\frac{3}{2}} (t) + M C_0  E^2(t). 
\end{aligned}
\end{equation}
Now suppose that 
\begin{equation} \label{eq55}
MC_0 (\|\nabla \mathbf{u}_0\|_{L^2}^2 + \|\nabla \mathbf{H}_0\|_{L^2}^2 ) \le \frac{1}{64(\mu+\nu)},
\end{equation}
which implies
\begin{equation} \label{eq545}
\begin{aligned}
MC_0 (\mu \|\nabla \mathbf{u}_0\|_{L^2}^2 + \nu \|\nabla \mathbf{H}_0\|_{L^2}^2 ) & \le MC_0 (\mu+\nu)  (\|\nabla \mathbf{u}_0\|_{L^2}^2 +  \|\nabla \mathbf{H}_0\|_{L^2}^2 ) \\
& \le \frac{1}{64(\mu+\nu)} \times (\mu + \nu) = \frac{1}{64}.
\end{aligned}
\end{equation}
And set
\begin{equation} \label{eq56}
T_* \triangleq \max\{t \in [0, T]: E(s) \le 16 (\mu\|\nabla \mathbf{u}_0\|_{L^2}^2 + \nu\|\nabla \mathbf{H}_0\|_{L^2}^2), \forall s \in (0, t)\}.
\end{equation}
We claim that
$$T_* = T.$$
Otherwise, we have $T_* \in (0, T).$ By continuity of $E(t),$ it follows from \eqref{eq54} and \eqref{eq545} that 
\begin{equation} \label{eq57}
\begin{aligned}
E(T_*) & \le 2(\mu\|\nabla \mathbf{u}_0\|_{L^2}^2 + \nu\|\nabla \mathbf{H}_0\|_{L^2}^2) +  \sqrt{MC_0} \cdot \sqrt{16 (\mu\|\nabla \mathbf{u}_0\|_{L^2}^2 + \nu \|\nabla \mathbf{H}_0\|_{L^2}^2)} E(T_*) \\
& \quad  + M C_0 \cdot 16(\mu\|\nabla \mathbf{u}_0\|_{L^2}^2 + \nu \|\nabla \mathbf{H}_0\|_{L^2}^2) E(T_*) \\
& = 2(\mu\|\nabla \mathbf{u}_0\|_{L^2}^2 + \nu\|\nabla \mathbf{H}_0\|_{L^2}^2) +  \sqrt{16 MC_0 (\mu\|\nabla \mathbf{u}_0\|_{L^2}^2 + \nu \|\nabla \mathbf{H}_0\|_{L^2}^2)} E(T_*) \\
&\quad + 16 M C_0 (\mu\|\nabla \mathbf{u}_0\|_{L^2}^2 + \nu \|\nabla \mathbf{H}_0\|_{L^2}^2) E(T_*) \\
&\le 2(\mu\|\nabla \mathbf{u}_0\|_{L^2}^2 + \nu\|\nabla \mathbf{H}_0\|_{L^2}^2) + \frac{3}{4}E(T_*),
\end{aligned}
\end{equation}
and thus 
$$E(T_*) \le 8 (\mu\|\nabla \mathbf{u}_0\|_{L^2}^2 + \nu\|\nabla \mathbf{H}_0\|_{L^2}^2),$$
which contradicts \eqref{eq56}.

Choosing $\varepsilon_0 = \frac{1}{64M(\mu+\nu)}$, by virtue of the claim we have showed in the above, we derive that
$$ E(t) \le 8(\mu\|\nabla \mathbf{u}_0\|_{L^2}^2 + \nu\|\nabla \mathbf{H}_0\|_{L^2}^2), \quad 0 <t < T, $$
provided that \eqref{eq53} holds true. This gives the desired \eqref{eq52} which consequently completes the proof of Lemma \ref{lem3}.
\end{proof}

Now we are ready to give a proof of Theorem \ref{thm2}.

\noindent {\bf{Proof of Theorem \ref{thm2}}.} \quad Let $\varepsilon_0$ be the constant stated in Lemma \ref{lem3} and suppose the initial data $(\rho_0, \mathbf{u}_0, \mathbf{H}_0, \theta_0)$ satisfies \eqref{RC}, \eqref{CC1}, \eqref{CC2}, and 
\begin{equation}
(\|\sqrt{\rho_0} \mathbf{u}_0\|_{L^2}^2 + \|\mathbf{H}_0\|_{L^2}^2)(\|\nabla \mathbf{u}_0\|_{L^2}^2 + \|\nabla \mathbf{H}_0\|_{L^2}^2)) \le \varepsilon_0.
\end{equation}

According to Lemma \ref{local}, there is a unique strong solution $(\rho, \mathbf{u}, \mathbf{H}, \theta)$ to the system \eqref{MHD}-\eqref{boundary}. Let $T^*$ be the maximal existence time to that solution. We will show that $T^* = \infty.$ Supposing, by contradiction, that $T^* < \infty,$ then by \eqref{bloweq}, we deduce that for any $(s,r)$ with $\frac{2}{s} + \frac{3}{r} \le 1, r >3,$ it holds that
$$\int_0^{T^*} \|\mathbf{u}\|_{L^r_{\omega}}^s dt = \infty,$$
which combined with the inequality $\|\mathbf{u}\|_{L^6_{\omega}}^4 \le \|\mathbf{u}\|_{L^6}^4 \le C \|\nabla \mathbf{u}\|_{L^2}^4$ leads to 
\begin{equation} \label{eq58}
\int_0^{T^*} \|\nabla \mathbf{u}\|_{L^2}^4 dt = \infty.
\end{equation}
By Lemma \ref{lem3}, for any $0<T<T^*$, it holds that
$$\sup_{0 \le t \le T} \|\nabla \mathbf{u}\|_{L^2}^2 \le 8(\mu\|\nabla \mathbf{u}_0\|_{L^2}^2 + \nu\|\nabla \mathbf{H}_0\|_{L^2}^2).$$
This together with \eqref{energy4} gives rise to
\begin{equation*}
\begin{aligned}
\int_0^{T^*} \|\nabla \mathbf{u}\|_{L^2}^4 dt & \le (\sup_{0 \le t \le T^*} \|\nabla \mathbf{u}\|_{L^2} ^2) \int_0^{T^*} \|\nabla \mathbf{u}\|_{L^2}^2 dt \\
& \le 8(\mu\|\nabla \mathbf{u}_0\|_{L^2}^2 + \nu\|\nabla \mathbf{H}_0\|_{L^2}^2) \times (2\mu)^{-1}C_0 < \infty,
\end{aligned}
\end{equation*}
which contradicts \eqref{eq58}. This contradiction implies that $T^* = \infty,$ and thus we obtain the global strong solution. Therefore the proof of Theorem \ref{thm2} is completed.


\begin{thebibliography}{20}
\bibitem{Bie}
Q. Bie; Q. Wang; Z. Yao, \textit{Global well-posedness of the 3D incompressible MHD equations with variable density}. Nonliear Anal. Real World Appl. 47, 2019, 85-105.
\bibitem{ChenG}
F. Chen; B. Guo; X. Zhai, \textit{Global solution to the 3-D inhomogeneous incompressible MHD system with discontinuous density}. Kinet. Relat. Models 12, 2019, 37-58.
\bibitem{ChenL}
F. Chen; Y. Li; H. Xu, \textit{Global solution to the 3D nonhomogeneous incompressible MHD equations with some large initial data}. Discrete Contin. Dyn. Syst. 36, 2016, 2945-2967.
\bibitem{ChenT}
Q. Chen; Z. Tan; Y. Wang, \textit{Strong solutions to the incompressible magnetohydrodynamic equations}. Math. Methods Appl. Sci. 34, 2011, 94-101.
\bibitem{ChoK}
Y. Cho; H. Kim, \textit{Existence results for viscous polytropic fluids with vacuum}. J. Differential Equations 228 (2006), no. 2, 377–411.
\bibitem{David}
P. Davidson, \textit{Introduction to magnetohydrodynamics.} 2nd edition. Cambridge University Press, 2017.
\bibitem{Fan}
J. Fan; W. Yu, \textit{Strong solution to the compressible magnetohydrodynamic equations with vacuum}. Nonlinear Anal. Real World Appl. 10 (2009), no. 1, 392–409.
\bibitem{Feireisl}
E. Feireisl, \textit{Dynamics of viscous compressible fluids}. Oxford Lecture Series in Mathematics and its Applications, 26. Oxford University Press, Oxford, 2004. xii+212 pp.
\bibitem{DLebris}
B. Desjardins; C. Le Bris, \textit{Remarks on a nonhomogeneous model of magnetohydrodynamics}. Differential Integral Equations 11 (1998), no. 3, 377–394. 
\bibitem{GLebris}
J. Gerbeau; C. Le Bris, \textit{Existence of solution for a density-dependent magnetohydrodynamic equation}. Adv. Differential Equations 2 (1997), no. 3, 427–452.
\bibitem{Gra}
L. Grafakos, \textit{Classical Fourier analysis}. Second edition. Graduate Texts in Mathematics, 249. Springer, New York, 2008. xvi+489 pp.
\bibitem{HeLi}
C. He; J. Li; B. Lü, \textit{Global well-posedness and exponential stability of 3D Navier-Stokes equations with density-dependent viscosity and vacuum in unbounded domains}. Arch. Ration. Mech. Anal. 239 (2021), no. 3, 1809–1835.
\bibitem{HeX}
C. He; Z. Xin, \textit{On the regularity of weak solutions to the magnetohydrodynamic equations}. J. Diff. Eq. 213 (2005), no. 2, 235-254.
\bibitem{HuangLi}
X. Huang; J. Li, \textit{Serrin-type blowup criterion for viscous, compressible, and heat conducting Navier-Stokes and magnetohydrodynamic flows.} Comm. Math. Phys. 324 (2013), no. 1, 147–171.
\bibitem{HuangW}
X. Huang; Y. Wang, \textit{Global strong solution to the 2D nonhomogeneous incompressible MHD system}. J. Differential Equations 254, 2013, 511-527.
\bibitem{K}
H. Kim, \textit{A blow-up criterion for the nonhomogeneous incompressible Navier-Stokes equations}. SIAM J. Math. Anal.  37  (2006),  no. 5, 1417-1434.
\bibitem{kozono}
H. Kozono; M. Yamazaki, \textit{Uniqueness criterion of weak solutions to the stationary Navier-Stokes equations in exterior domains}. Nonlinear Anal. 38 (1999), no. 8, Ser. A: Theory Methods, 959–970.
\bibitem{lions}
P. L. Lions, \textit{Mathematical topics in fluid mechanics. Vol. 1. Incompressible models}. Oxford University Press, Oxford, 1996. xiv+237 pp.
\bibitem{Nirenberg}
L. Nirenberg, \textit{On elliptic partial differential equations}. Ann. Scuola Norm. Sup. Pisa Cl. Sci. (3) 13 (1959), 115–162.
\bibitem{Sohr}
H. Sohr, \textit{The Navier-Stokes equations. An elementary functional analytic approach}. Birkh\"auser Advanced Texts, Birkh\"auser, Basel, 2001. 
\bibitem{WangY}
Y. Wang, \textit{Weak Serrin-type blowup criterion for three-dimensional nonhomogeneous viscous incompressible heat conducting flows}. Phys. D 402 (2020), 132203, 8 pp.
\bibitem{WangYu}
W. Wang; H. Yu; P. Zhang, \textit{Global strong solutions for 3D viscous incompressible heat conducting Navier-Stokes flows with the general external force}. Math. Methods Appl. Sci. 41, 2018, 4589-4601.
\bibitem{Wu}
H. Wu, \textit{Strong solutions to the incompressible magnetohydrodynamic equations with vacuum}. Comput. Math. Appl. 61 (2011), no. 9, 2742–2753.
\bibitem{Zhong2}
X. Zhong, \textit{Global strong solution for 3D viscous incompressible heat conducting Navier-Stokes flows with non-negative density}. J. Differential Equations 263 (2017), no. 8, 4978–4996.
\bibitem{Zhong5}
X. Zhong, \textit{Global strong solutions for 3D viscous incompressible heat conducting magnetohydrodynamic flows with non-negative density}. J. Math. Anal. Appl. 446 (2017), no. 1, 707–729.
\bibitem{Zhong3}
X. Zhong, \textit{Global well-posedness and exponential decay of 2D nonhomogeneous Navier-Stokes and magnetohydrodynamic equations with density-dependent viscosity and vacuum}. J. Geom. Anal. 32 (2022), no. 1, Paper No. 19, 26 pp.
\bibitem{Zhong4}
X. Zhong, \textit{Global well-posedness to the 2D Cauchy problem of nonhomogeneous heat conducting magnetohydrodynamic equations with large initial data and vacuum}. Calc. Var. Partial Differential Equations 60 (2021), no. 2, Paper No. 64, 24 pp.
\bibitem{Zhong1}
X. Zhong, \textit{Global existence and large time behavior of strong solutions to the nonhomogeneous heat conducting magnetohydrodynamic equations with large initial data and vacuum}. Anal. Appl. (Singap.) 20 (2022), no. 2, 193–219.
\bibitem{Zhou}
L. Zhou, \textit{Serrin-type blowup criterion of three-dimensional nonhomogeneous heat conducting magnetohydrodynamic flows with vacuum}. Electron. J. Qual. Theory Differ. Equ. 2019, Paper No. 81, 16 pp.
\bibitem{ZhuO}
M. Zhu; M. Ou, \textit{Global strong solutions to the 3D incompressible heat-conducting magnetohydrodynamic flows}. Math. Phys. Anal. Geom. 22 (2019), no. 1, Paper No. 8, 17 pp.


\end{thebibliography}
\end{document}